\documentclass{article}
\usepackage{amsmath,amssymb,amsthm,enumerate}
\usepackage[all]{xy}
\usepackage{graphics}
\usepackage{hyperref}

\newtheorem{Lemma}{Lemma}[section]
\newtheorem{Proposition}[Lemma]{Proposition}

\theoremstyle{definition}

\newtheorem{Theorem}[Lemma]{Theorem}
\newtheorem{Definition}[Lemma]{Definition}
\newtheorem{Fact}[Lemma]{Fact}

\newtheorem{Remark}[Lemma]{Remark}
\newtheorem{Corollary}[Lemma]{Corollary}

\makeatletter
\@addtoreset{equation}{section} 
\makeatother

\title{Existence of a conjugate point in the incompressible Euler flow on a three-dimensional ellipsoid}
\author{L. A. Lichtenfelz\thanks{
Department of Mathematics, 
Wake Forest University,
127 Manchester Hall Box 7388, Winston-Salem, NC 27109,
America,
E-mail address: lichtel@wfu.edu}
\and
T. Tauchi\thanks{
Institute of Mathematics for Industry, Kyushu University, Nishi-ku, Fukuoka, 819-0395, Japan, E-mail address: tauchi.taito.342@m.kyushu-u.ac.jp} 
\and T. Yoneda\thanks{
Graduate School of Economics, Hitotsubashi University, 2-1
Naka, Kunitachi, Tokyo 186-8601, Japan,
E-mail address: t.yoneda@r.hit-u.ac.jp}}
\date{}

\begin{document}
\maketitle
\begin{abstract}
The existence of 
a conjugate point on the volume-preserving diffeomorphism group
of a compact Riemannian manifold $M$
is related to the Lagrangian stability
of a solution of 
the incompressible Euler equation on $M$.
The Misio{\l}ek curvature 
is a reasonable criterion for the existence 
of a conjugate point on the volume-preserving diffeomorphism group
corresponding to a stationary solution of 
the incompressible Euler equation.
In this article,
we introduce a class of stationary solutions
on an arbitrary Riemannian manifold
whose behavior is nice with respect to 
the Misio{\l}ek curvature
and give a positivity result of the Misio{\l}ek curvature
for solutions belonging to this class.
Moreover,
we also show the 
existence of a conjugate point in the three-dimensional ellipsoid case
as its corollary.
\end{abstract}
{\bf MSC2020;} Primary 35Q35; Secondary 35Q31.

\section{Main result}
Let $(M,g)$ be a compact $n$-dimensional Riemannian manifold without boundary. 
Consider the incompressible Euler equation on $M$:
\begin{eqnarray}
	\frac{\partial u}{\partial t}+\nabla_{u}u
	&=&
	-\operatorname{grad}p,
	\nonumber\\
	\operatorname{div}u
	&=&
	0,
	\label{E-eq-intro}\\
	u|_{t=0}&=&u_{0}.
	\nonumber
\end{eqnarray}
Let $u$ be a stationary solution of this equation
and $v$ a divergence-free vector field on $M$.
Then,
the Misio{\l}ek curvature ${\mathfrak m}{\mathfrak c}_{u,v}$ is defined by
\begin{equation}\label{mc_criterion}
	{\mathfrak m}{\mathfrak c}_{u,v}
	:=
	-
	\int_{M}
	g([u,v],[u,v])
	\mu
	-
	\int_{M}
	g(u,[[u,v],v])
	\mu,
\end{equation}
where $\mu$ is the volume form on $M$ 
(see \cite[Lem.~B.6]{TY2} or Appendix \ref{Appendix-diffomorphism-and-MC}).
The importance of this functional  
is 
that ${\mathfrak m}{\mathfrak c}_{u,v}>0$ implies 
that the solution $u$
contains a conjugate point when viewed as a geodesic in the group ${\mathcal D}^{s}_{\mu}(M)$ of volume-preserving Sobolev $H^{s}$ diffeomorphisms of $M$ starting at the identity \cite[Fact.~1.1]{TY1}.

The question of whether geodesics in ${\mathcal D}^{s}_{\mu}(M)$ contain conjugate points goes back to Arnold \cite{A}, \cite{AK}. Examples have been found by multiple authors, e.g., \cite{Mstability}, \cite{MconjT2}, \cite{Shnirelman}, \cite{Preston-conj}, \cite{PW}, \cite{Benn-MC}, \cite{Benn-coad}, \cite{Benn-sympl} among others, using several different techniques.

Conjugate points are related to Lagrangian stability of the corresponding flow (\cite{Mstability}, \cite{Pstability}, \cite{TY3},  \cite{Benn-coad}, \cite{DMSY}, \cite{Benn-MC}). They can also be used to obtain detailed local information about the data-to-solution map of the Cauchy problem \eqref{E-eq-intro} in Lagrangian coordinates \cite{MLittauer},  \cite{L}.

Despite the aforementioned efforts, we still lack a good understanding of the nature and structure of the set of conjugate points. This motivates our search for new examples.

The criterion \eqref{mc_criterion} was first used in \cite{MconjT2} by G. Misio{\l}ek and recently attracted attention again \cite{DMSY,TY1,TY2,TY3}. A variant of this criterion was used in \cite{Benn-MC} to construct examples of conjugate points along non-stationary geodesics.

In particular,
the second and third authors considered two-dimensional ellipsoids,
and showed the positivity of the Misio{\l}ek curvature 
corresponding to almost every zonal flow $Z$ on them in \cite{TY1},
motivated by the existence of stable multiple zonal jet flow on Jupiter,
whose mechanism is not yet well understood.

The first aim of this article was to generalize the result of \cite{TY1} to the three-dimensional ellipsoid case.
However, the result of \cite{TY1} concerns zonal flows, which are only defined on two-dimensional spheres  or ellipsoids. Thus, in order to generalize the result of \cite{TY1}, we give a new definition of a zonal flow on an arbitrary Riemannian manifold (cf. Remark \ref{rmk_2} below).

Unexpectedly, this definition behaves nicely with respect to the Misio{\l}ek curvature
and  we  obtain the following criteria for the positivity of the Misio{\l}ek curvature on an arbitrary compact Riemannian manifold. See Definitions \ref{Def-zonal-flow-arbitrary}, \ref{Def-zonal-geodesic-arbitrary}, \ref{Def-zonal-positive},  and \ref{Def-zonal-S1} for the meaning of a non-geodesic positive $S^{1}$-zonal flow.

\begin{Theorem}
\label{Theorem-Main}
	Let $Z$ be a non-geodesic positive $S^{1}$-zonal flow on a compact Riemannian manifold $M$ with $\dim M\geq 3$.
	Then,
	there exists a divergence-free vector field $Y$ on $M$
	satisfying
	$$
	{\mathfrak m}{\mathfrak c}_{Z,Y}>0.
	$$
\end{Theorem}

\begin{Remark}\label{rmk_2}
Let $M$ be an arbitrary Riemannian manifold with $\dim M=3$.
	In \cite{LMP-Axi},
	the first author, G. Misio{\l}ek, and S. C. Preston considered the space of axisymmetric vector fields $u$ on $M$, which means that
	$u$ satisfies $[u,X]=0$ for some fixed Killing vector field $X$ on $M$ \cite[Sect. 3]{LMP-Axi}.
	If $Z=fX$ is a zonal flow in the sense of Definition \ref{Def-zonal-flow-arbitrary},
	we have in particular that 
	\begin{eqnarray*}
		[Z,X]=[fX,X]=f[X,X]-X(f)X=0
	\end{eqnarray*}
	by
	Lemma \ref{Lemma-zonal-div=0}.
	This implies that any zonal flow can be regarded as an axisymmetric vector field on $M$.
	By Theorem \ref{Theorem-Main}, 
	any non-geodesic positive $S^{1}$-zonal flow $Z$ on $M$
	has a conjugate point on ${\mathcal D}^{s}_{\mu}(M)$
	by the M-criterion (Fact \ref{Fact-M-criterion})
	if $M$ is compact.

\end{Remark}

As a corollary of theorem \ref{Theorem-Main}, we have the following generalization of the result in \cite{TY1} to the case of three-dimensional ellipsoids. Let
$$
M_{a}^{3}:=\{
(x,y,z,w)\in{\mathbb R}^{4}\mid 
x^{2}+y^{2}=a^{2}(1-z^{2}-w^{2})
\}
$$
for $a>0$.
Define the subsets of $M_{a}^{3}$
by
\begin{eqnarray*}
	N:=
	\{
(x,y,0,0)\in{\mathbb R}^{4}\mid 
x^{2}+y^{2}=a^{2}
\},
\quad
S:=\{
(0,0,z,w)\in{\mathbb R}^{4}\mid 
z^{2}+w^{2}=1
\}.
\end{eqnarray*}
\begin{Theorem}
	\label{Theorem-Main-2}
	Let
	$Z$ be a non-geodesic $S^{1}$-zonal flow on $M_{a}^{3}$
	with $\operatorname{supp}(Z)\subset M^{3}_{a}\backslash (N\cup S)$.
	Then,
	there exists a divergence-free vector field $Y$ on $M_{a}$
	such that 
	$$
	{\mathfrak m}{\mathfrak c}_{Z,Y}>0.
	$$
\end{Theorem} 
\begin{Remark}
	Any zonal flow on the two-dimensional sphere or an ellipsoid 
	in the sense of \cite{TY1}
	is a non-geodesic  $S^{1}$-zonal flow in the sense of this article (see Corollary \ref{Cor-any-zonal-is-S1-geodesic-on-2D-ellipsoid} and Remark \ref{Remark-S2-zonal}).
	Thus, Theorem \ref{Theorem-Main-2}
	can be regard as a generalization of \cite[Thm.~1.2]{TY1}.
\end{Remark}

Theorem \ref{Theorem-Main-2} implies 
the existence of a conjugate point
on the volume-preserving diffeomorphism group
${\mathcal D}^{s}_{\mu}(M^{3}_{a})$
of $M^{3}_{a}$.
See Appendix \ref{Appendix-diffomorphism-and-MC}
for unexplained notations
and details.

\begin{Corollary}
\label{Cor-Main-3}
Let $s>\frac{7}{2}$,
$Z$ be a non-geodesic $S^{1}$-zonal flow on $M_{a}^{3}$
	with $\operatorname{supp}(Z)\subset M^{3}_{a}\backslash (N\cup S)$,
	and 
	$\eta(t)$ a geodesic on 
	${\mathcal D}^{s}_{\mu}(M^{3}_{a})$ corresponding $Z$.
	Then,
	there exists a point conjugate to the identity element 
	$e\in {\mathcal D}^{s}_{\mu}(M)$ along $\eta(t)$.
\end{Corollary}
\begin{proof}
	This follows from Theorem \ref{Theorem-Main-2}
	by the M-criterion (Fact \ref{Fact-M-criterion}).
\end{proof}

\begin{Remark}
A different criterion for the existence of conjugate points was derived by S. Preston in \cite[Thm.~3.1 and (3.32)]{Preston-conj}. It differs from the $M$-criterion in that it is specific to the 3D case, and relies on finding nontrivial solutions to an ODE localized to some point $p \in M$. That criterion can also be used to prove the existence of the conjugate points provided here. Nevertheless, our examples had never been found before, to the best of our knowledge.
\end{Remark}

This paper is organized as follows.
In Section \ref{Section-generalization-zonal},
we propose a definition of a zonal flow on an
arbitrary Riemannian manifold 
and investigate its properties.
In Section \ref{Section-Misiolek-zonal},
we calculate the Misio{\l}ek curvature of 
a zonal flow
and 
prove a sufficient condition for the positivity
of the Misio{\l}ek curvature.
In Section \ref{Section-2D-ellipsoid},
we show that 
the definition of a zonal flow given in Section \ref{Section-generalization-zonal}
is equivalent to that of \cite{TY1}
in the two-dimensional ellipsoid case.
In Section \ref{Section-class-zonal-3D-ellip},
we classify zonal flows on 
three-dimensional ellipsoids.
In Section \ref{Section-prove-Theorem-Main-2},
we prove Theorem \ref{Theorem-Main-2}
as a corollary of Theorem \ref{Theorem-Main}.
In Appendixes \ref{Appendix-proof-Lemma-S2-Killing} and \ref{Appendix-proof-of-Lemma-Y-existence},
we prove some elementary lemmas used in this article.
In Appendix \ref{Appendix-diffomorphism-and-MC},
we summarize
the background materials
of the Misio{\l}ek curvature.

\section*{Acknowledgments}
The research of TT was partially supported by Grant-in-Aid for JSPS Fellows (20J00101), Japan Society for the Promotion of Science (JSPS). The research of TY was partially supported by Grant-in-Aid for Scientific Research B (17H02860, 18H01136, 18H01135 and 20H01819), Japan Society for the Promotion of Science (JSPS).

\section{A generalization of  zonal flow}
\label{Section-generalization-zonal}
In this section,
we propose a definition of a zonal flow
on an arbitrary Riemannian manifold
and 
investigate its properties.

\subsection{Killing vector fields}
In this section,
we recall the basic properties of the Killing vector fields.
Therefore,
all the materials in this section are well known.
However, we prove some results for convenience.
For example, see \cite[Sect.~3]{Riemannian} for more details.

Recall that 
a vector field $X$ on a Riemannian manifold $(M,g)$
is Killing
if and only if 
\begin{equation}
	g(\nabla_{V}X,W)
	=
	-
	g(\nabla_{W}X,V)
	\label{eq-Killing-Def-VW}
\end{equation}
for any vector fields $V,W$ on $M$,
where $\nabla$ is the Levi-Civita connection.
For the notational simplicity,
we set
\begin{equation}
	\|X\|^{2}:=g(X,X),
	\qquad
	\|X\|:=\sqrt{\|X\|^{2}}.
\end{equation}

\begin{Lemma}
\label{Lemma-Killing-nabla}
	Let $X$ be a Killing vector field on a Riemannian manifold $M$.
	Then,
	we have
	\begin{equation}
		2\nabla_{X}X=-\operatorname{grad}(\|X\|^{2}).
	\end{equation}
	In particular,
	$\nabla_{X}X=0$ if and only if
	$\|X\|$ is a constant function.
\end{Lemma}
\begin{proof}
	It is sufficient to prove 
\begin{equation}
2g(\nabla_{X}X,W)=-g(\operatorname{grad}(\|X\|^{2}),W)
\label{eq-Lemma-2nablaXX=gradX}
\end{equation}
for any vector field $W$ on $M$
by the nondegeneracy of the Riemannian metric $g$.
However,
	we have
	\begin{eqnarray}
		g(\operatorname{grad}(\|X\|^{2}),W)
		&=&
		W(\|X\|^{2})
	\nonumber\\
	&=&
	2g(\nabla_{W}X,X).
	\nonumber
	\end{eqnarray}
	Therefore,
	\eqref{eq-Lemma-2nablaXX=gradX} follows from
	\eqref{eq-Killing-Def-VW}.
\end{proof}

\begin{Lemma}
	\label{Lemma-Killing-XXX=0}
	Let $X$ be a Killing vector field on a Riemannian manifold $M$.
	Then,
	we have
	\begin{eqnarray}
		X(\|X\|^{2})=0.
		\label{eq-Lemma-XXX=0}
	\end{eqnarray}
\end{Lemma}
\begin{proof}
	By Lemma \ref{Lemma-Killing-nabla},
	we have
	\begin{eqnarray*}
		X(\|X\|^{2})&=&
		g(\operatorname{grad}(\|X\|^{2}),X)
		\\
		&=&
		-2
		g(\nabla_{X}X,X).
	\end{eqnarray*}
Thus,
\eqref{eq-Lemma-XXX=0} follows from
	\eqref{eq-Killing-Def-VW}.
\end{proof}

\begin{Lemma}
\label{Lemma-unique-Kiiling}
	Let $M$ be a connected Riemannian manifold
	and
	$X_{1}$, $X_{2}$ 
	Killing vector fields on $M$.
	Then,
	if there exists an open subset $U\subset M$
	on which $X_{1}=X_{2}$,
	then we have $X_{1}=X_{2}$ on $M$.
	In particular, the zero set 
	$\{x\in M\mid X=0\}$
	of any nonzero Killing vector field $X$ on $M$
	dose not have an interior point.
\end{Lemma}

Before the proof,
recall the following fact.
Let ${\mathfrak X}(M)$
be the space of vector fields on $M$.
For $V\in{\mathfrak X}(M)$,
we write $\nabla V$ 
for an operator 
\begin{eqnarray}
	\nabla V:{\mathfrak X}(M)&\to & {\mathfrak X}(M)
	\label{eq-Def-nablaX-ope}\\
W&\mapsto& \nabla V(W):=\nabla_{W}V.
\nonumber
\end{eqnarray}

\begin{Fact}[{\cite[Lem.~3]{Nomizu60}}]
	\label{Fact-unique-Kiiling}
	Let $M$ be a Riemannian manifold.
	Then,
	a Killing vector field
	defined on a connected open subset $U\subset M$
	is uniquely determined by the values of $X$
	and $\nabla X$
	at any single point of $U$.
\end{Fact}

\begin{proof}[Proof of Lemma \ref{Lemma-unique-Kiiling}]
	Note that if $X_{1}=X_{2}$ on $U$,
	we have $\nabla X_{1}=\nabla X_{2}$ on $U$
	which implies
	$X_{1}=X_{2}$ on $M$
	by
	Fact \ref{Fact-unique-Kiiling}.
	Moreover,
	if 
	$\{x\in M\mid X=0\}$ 
	has an interior point,
	there exists an open subset $U\subset M$
	on which we have $X=0$.
	This completes the proof
	because the zero vector fields is Killing.
\end{proof}

\subsection{A definition of zonal flow}
In this section,
we propose a definition of a zonal flow 
on an arbitrary Riemannian manifold $M$.
Let $\mu$ be the volume form on $M$
and
\begin{eqnarray*}
P:{\mathfrak X}(M)&\to &{\mathfrak X}_{\mu}(M),
\\
V&\mapsto & P(V):=V-\operatorname{grad}f
\end{eqnarray*}
be the $L^{2}$ orthogonal projection,
where ${\mathfrak X}_{\mu}(M)$
is the space of divergence-free vector fields on $M$,
$f:=\Delta^{-1}\operatorname{div}(V)$.
Thus,
for $V\in{\mathfrak X}(M)$,
$P(V)=0$
if and only if
we have
\begin{equation}
	V=-\operatorname{grad} p
\end{equation}
for some function $p$ on $M$.
\begin{Definition}
\label{Def-zonal-flow-arbitrary}
Let $M$ be an arbitrary Riemannian manifold $M$
and
$Z$ a vector field on $M$.
Then we call $Z$ a \textit{zonal flow}
if it satisfies the following three conditions.
\begin{enumerate}[(1)]
	\item There exist a function $f$ and Killing vector field $X$ on $M$ such that $Z=fX$,
	\item $\operatorname{div}(Z)=0$,
	\item $P(\nabla_{Z}Z)=0$.
\end{enumerate}
\end{Definition}

Then,
the following is obvious.
\begin{Proposition}
\label{Prop-zonal-stationary}
	Let
	$M$ be a compact Riemannian manifold 
	and
	$Z$ be zonal flow on $M$.
	Then,
	$Z$ is a stationary solution of the incompressible 
	Euler equation \eqref{E-eq-intro} on $M$.
\end{Proposition}
\begin{proof}
	This is a consequence of the definition.
\end{proof}

We list up some elementary properties of zonal flows.

\begin{Lemma}
\label{Lemma-zonal-div=0}
	Let $Z$ be a vector field on a Riemannian manifold $M$.
	Suppose $Z=fX$ with a function $f$ and a Killing vector field $X$ on $M$.
	Then $\operatorname{div}(Z)=0$ if and only if $X(f)=0$.
\end{Lemma}
\begin{proof}
	We have
	\begin{eqnarray}
		\operatorname{div}(fX)=
		X(f)+ f\operatorname{div}(X).
	\end{eqnarray}
	Because any Killing vector field is divergence-free,
	 we have the lemma.
\end{proof}

\begin{Lemma}
\label{Lemma-zonal-nabla-grad}
	Let $Z=fX$ with a nonzero function $f$ and a Killing vector field $X$ on a Riemannian manifold $M$.
	Suppose $\operatorname{div}(Z)=0$.
	Then, we have
	\begin{eqnarray}
		\nabla_{Z}Z
		=
		f^{2}\nabla_{X}X
		=
		-\frac{f^{2}}{2}\operatorname{grad}(\|X\|^{2}).
		\label{eq-Lemma-nablaZZ}
	\end{eqnarray}
\end{Lemma}
\begin{proof}
The assumption $\operatorname{div}(Z)=0$
and Lemma \ref{Lemma-zonal-div=0}
imply
\begin{eqnarray*}
	\nabla_{Z}Z&=&f\nabla_{X}(fX)
	\\
	&=&
	f^{2}\nabla_{X}X.
\end{eqnarray*}
This implies \eqref{eq-Lemma-nablaZZ}
by Lemma \ref{Lemma-Killing-nabla}.
\end{proof}

\begin{Corollary}
\label{Cor-geodesic-equiv}
	Let $Z=fX$ with a nonzero function $f$ and a Killing vector field $X$ on a Riemannian manifold $M$.
	Suppose $\operatorname{div}(Z)=0$.
	Then,
	the following are equivalent:
	\begin{enumerate}[(i)]
		\item $\nabla_{Z}Z=0$,
		\item $\nabla_{X}X=0$,
		\item $\|X\|$ is a constant function.
	\end{enumerate}
\end{Corollary}
\begin{proof}
Note that
$\operatorname{grad}(\|X\|^{2})=0$
if and only if $\|X\|$ is a constant function.
Thus,
this corollary follows from 
Lemma \ref{Lemma-zonal-nabla-grad}.
\end{proof}

According to this corollary,
we take the following definition.
\begin{Definition}
\label{Def-zonal-geodesic-arbitrary}
Let $M$ be a Riemannian manifold and
	$Z$ a zonal flow
	on $M$.
	Then,
	we call $Z$ a \textit{geodesic zonal flow}
	if $\nabla_{Z}Z=0$.
	We also say that 
	a zonal flow $Z$ is non-geodesic
	if $Z$ is not a geodesic zonal flow.
\end{Definition}

\begin{Lemma}
\label{Lemma-geodesic-zonal-fZ}
	Let $X$ be a Killing vector field on $M$
	such that
	$\nabla_{X}X=0$.
	Then,
	$Z=fX$ is a geodesic zonal flow on $M$
	for any function $f$ satisfying $X(f)=0$.
\end{Lemma}
\begin{proof}
	The assumption $X(f)=0$ and Lemma \ref{Lemma-zonal-div=0}
	imply
	$\operatorname{div}(Z)=0$.
	The assumption of $\nabla_{X}X=0$ implies
	$\nabla_{Z}Z=0$
	by Lemma \ref{Lemma-zonal-nabla-grad}.	
	This completes the proof.
\end{proof}

\subsection{Further properties of zonal flows}
In this section,
we state further properties of zonal flows.
The first is the uniqueness of 
a representation of $Z=fX$
with a function $f$ and a Killing vector field $X$.

\begin{Lemma}
\label{Lemma-zonal-rep-unique}
	Let $Z$ be a nonzero zonal flow on a connected Riemannian manifold $M$.
	Represent $Z=fX$ with a function $f$ and Killing vector field $X$ on $M$.
	Then, $f$ and $X$ are unique up to constant multiple.
	In other words,
	if we have $Z=f_{1}X_{1}=f_{2}X_{2}$
	with functions $f_{1},f_{2}$ and Killing vector fields $X_{1},X_{2}$ on $M$,
	then there exists a constant $C\in{\mathbb R}$
	such that 
	$f_{1}=Cf_{2}$
	and
	$CX_{1}=X_{2}$.
\end{Lemma}

\begin{proof}
By Lemma \ref{Lemma-unique-Kiiling},
it is sufficient to prove that
there exists a nonempty open subset $U\subset M$
on which we have $X_{1}=CX_{2}$.
By the nonzero assumption of $Z$,
$$
U:=\{x\in M\mid f_{1}(x)\neq 0\}=\{x\in M\mid f_{2}(x)\neq 0\}
$$
is nonempty.
Setting $h:=\frac{f_{1}}{f_{2}}$,
we have
\begin{equation}
X_{1}=hX_{2}
\label{eq-X1=hX2}
\end{equation}
on $U$.
We will show $h$ is constant on $U$. 
Because $X_{1}$ is Killing,
we have
\begin{equation}
	0=\operatorname{div}(X_{1})=\operatorname{div}(hX_{2}),
\end{equation}
which implies
\begin{equation}
	X_{2}(h)=0
	\label{eq-Lemma-X2h=0}
\end{equation}
by Lemma \ref{Lemma-zonal-div=0}.

On the other hand,
because $X_{1}$ is Killing,
we have
\begin{eqnarray}
	g(\nabla_{V}X_{1},W)=-g(\nabla_{V}X_{1},W)
	\label{eq-gVX1W=-gWX1V}
\end{eqnarray}	 
for any vector fields $V,W$ on $M$.
Substituting  
\eqref{eq-X1=hX2}
to 
\eqref{eq-gVX1W=-gWX1V},
we have
\begin{eqnarray}
	V(h)g(X_{2},W)
	+
	h
	g(\nabla_{V}X_{2},W)
	=
	-W(h)g(X_{2},V)
	-
	hg(\nabla_{W}X_{2},V)
	\label{eq-VX2WhgnablaX2W=}
\end{eqnarray}
on $U$.
Because $X_{2}$ is also Killing,
we have
\begin{eqnarray}
	g(\nabla_{V}X_{2},W)=-g(\nabla_{W}X_{2},V).
	\label{eq-gVX2W=-gWX2V}
\end{eqnarray}	 
Substituting 
\eqref{eq-gVX2W=-gWX2V}
to
\eqref{eq-VX2WhgnablaX2W=},
we have
\begin{eqnarray}
	V(h)g(X_{2},W)
	=
	-W(h)g(X_{2},V).
\end{eqnarray}
Taking $W=X_{2}$,
we have
\begin{eqnarray}
	V(h)\|X_{2}\|^{2}
	&=&
	-X_{2}(h)g(X_{2},V).
\end{eqnarray}
Therefore,
\eqref{eq-Lemma-X2h=0} implies
\begin{eqnarray}
	V(h)\|X_{2}\|^{2}
	=
	0
	\label{eq-V(h)X_{2}^{2}=0}
\end{eqnarray}
for any vector field $V$ on $U$.
We note that the zero set of $X_{2}$
\begin{equation}
	\{x\in M\mid \|X_{2}\|=0\}
\end{equation}
does not have interior point
by Lemma \ref{Lemma-unique-Kiiling}.
Therefore,
\eqref{eq-V(h)X_{2}^{2}=0} implies
\begin{equation}
	V(h)=0
\end{equation}
for any vector field $V$ on $U$.
Thus,
$h$ is constant on $U$.
This completes the proof.
\end{proof}

The second is a necessary condition
on a function $f$ for $Z=fX$ to be a zonal flow.

\begin{Lemma}
\label{Lemma-zonal-grad-collinear}
	Let 
	$Z=fX$ be a zonal flow on a Riemannian manifold $M$
	and
	$U_{0}:=\{x\in M\mid \operatorname{grad}(\|X\|^{2})\neq 0\}$.
	Then,
	$\operatorname{grad}(f^{2})$ and 
	$\operatorname{grad}(\|X\|^{2})$
	are collinear on $U_{0}$.
	More precisely,
	there exists a function $F$
	up to positive constant multiple 
	on $U_{0}$
	satisfying
	\begin{eqnarray}
		\operatorname{grad}(f^{2})
		&=&
		F\operatorname{grad}(\|X\|^{2}).
		\label{eq-Lemma-gradf^2=Fgrad}
	\end{eqnarray}
	Moreover,
	we have
	\begin{eqnarray}
		X(F)=0.
		\label{eq-Lemma-XF=0}
	\end{eqnarray}
\end{Lemma}
\begin{proof}
The uniqueness assertion
follows from 
\eqref{eq-Lemma-gradf^2=Fgrad}
because 
$f$ and $X$ are unique up to constant multiple
by Lemma \ref{Lemma-zonal-rep-unique}.
Because $Z$ is a zonal flow,
there 
exists a function $p$ on $M$
satisfying $\nabla_{Z}Z=-\operatorname{grad}p$.
This equation is equivalent to
\begin{equation}
\label{eq-Lemma-ffgrad=gradp}
	\frac{f^{2}}{2}\operatorname{grad}(\|X\|^{2})
	=
	\operatorname{grad}p
\end{equation}
by Lemma \ref{Lemma-zonal-nabla-grad}.
We write $\operatorname{Hess}_{p}$
	for the Hessian of $p$,
	namely,
	\begin{equation}
	\operatorname{Hess}_{p}(V,W)
	:=
		g(\nabla_{V}\operatorname{grad}(p),W)
	\end{equation}
	for any vector field $V,W$ on $M$.
	By \eqref{eq-Lemma-ffgrad=gradp},
	we have
	\begin{eqnarray}
		2\operatorname{Hess}_{p}(V,W)
		&=&
		g(\nabla_{V}(f^{2}\operatorname{grad}(\|X\|^{2})),W)
		\nonumber
		\\
		&=&
		V(f^{2})
		g(\operatorname{grad}(\|X\|^{2}),W)
		+
		f^{2}
		g(
		\nabla_{V}
		\operatorname{grad}(\|X\|^{2}),W)
		\nonumber
		\\
		&=&
		V(f^{2})
		g(\operatorname{grad}(\|X\|^{2}),W)
		+
		f^{2}
		\operatorname{Hess}_{\|X\|^{2}}(V,W).
		\label{eq-Lemma-Hess=gVf}
	\end{eqnarray}
	Recall that the Hessian is always symmetric,
namely,
we have
\begin{equation}
	\operatorname{Hess}_{p}(V,W)=\operatorname{Hess}_{p}(W,V),
	\quad
	\operatorname{Hess}_{\|X\|^{2}}(V,W)=\operatorname{Hess}_{\|X\|^{2}}(W,V).
\end{equation}
Thus,
\eqref{eq-Lemma-Hess=gVf}
implies
\begin{equation}
	V(f^{2})
		g(\operatorname{grad}(\|X\|^{2}),W)
		=
		W(f^{2})
		g(\operatorname{grad}(\|X\|^{2}),V).
\end{equation}
Therefore,
we have
\begin{equation}
		g(\operatorname{grad}(f^{2}),V)
		g(\operatorname{grad}(\|X\|^{2}),W)
		=
		g(\operatorname{grad}(f^{2}),W)
		g(\operatorname{grad}(\|X\|^{2}),V)
\end{equation}
for any vector field $V,W$ on $M$.
This equation 
easily
implies that 
the orthogonal complements
of 
$\operatorname{grad}(f^{2})$
and
$\operatorname{grad}(\|X\|^{2})$
are equal
on 
$U_{0}\cap \{x\in M \mid \operatorname{grad}(f^{2})\neq 0\}$,
where $U_{0}:=\{x\in M\mid \operatorname{grad}(\|X\|^{2})\neq 0\}$.
This completes the proof of \eqref{eq-Lemma-gradf^2=Fgrad}.

By \eqref{eq-Lemma-gradf^2=Fgrad},
we have
\begin{eqnarray*}
	\operatorname{Hess}_{f^{2}}(X,W)
	&=&
	g(\nabla_{X}\operatorname{grad}(f^{2}),W)
	\\
	&=&
	g(\nabla_{X}F\operatorname{grad}(\|X\|^{2}),W)
	\\
	&=&
	X(F)
	g(\operatorname{grad}(\|X\|^{2}),W)
	+
	F
	g(\nabla_{X}\operatorname{grad}(\|X\|^{2}),W)
	\\
	&=&
	X(F)W(\|X\|^{2})
	+
	F
	\operatorname{Hess}_{\|X\|^{2}}(X,W)
\end{eqnarray*}
for any vector field $W$ on $U_{0}=\{x\in M\mid \operatorname{grad}(\|X\|^{2})\neq 0\}$.
On the other hand,
\begin{eqnarray*}
	\operatorname{Hess}_{f^{2}}(X,W)
	&=&
	\operatorname{Hess}_{f^{2}}(W,X)
	\\
	&=&
	W(F)X(\|X\|^{2})
	+
	F
	\operatorname{Hess}_{\|X\|^{2}}(W,X)
	\\
	&=&
	F\operatorname{Hess}_{\|X\|^{2}}(X,W).
\end{eqnarray*}
by Lemma \ref{Lemma-Killing-XXX=0}.
These imply
\begin{eqnarray}
	X(F)W(\|X\|^{2})=0
	\label{eq-Lemma-XFWXX=0}
\end{eqnarray}
for any vector field $W$ on $U_{0}$.
Because $\|X\|^{2}$ is non-constant on $U_{0}$,
\eqref{eq-Lemma-XFWXX=0} implies
\eqref{eq-Lemma-XF=0}.
\end{proof}

It turns out that the signature of $F$ is important 
in Section \ref{Section-sufficient-MC>0}.
Therefore we make the following definition.
\begin{Definition}
\label{Def-zonal-signature}
	Let $Z$ be a zonal flow on a Riemannian manifold $M$.
	Then,
	we define $\operatorname{sgn}(Z):M\to \{-1,0,1\}$ by
	\begin{equation}
		\operatorname{sgn}(Z):=
		\begin{cases}
			\operatorname{sgn}(F) & \text{on }U_{0}:=\{
			x\in M\mid \operatorname{grad}(\|X\|^{2})\neq 0
			\}\\
			0& \text{on }M\backslash U_{0},
		\end{cases}
	\end{equation}
	where $F$ is a function defined in Lemma \ref{Lemma-zonal-grad-collinear}
	and $\operatorname{sgn}:{\mathbb R}\to\{-1,0,1\}$.
\end{Definition}

\begin{Definition}
\label{Def-zonal-positive}
	Let $Z$ be a zonal flow on a Riemannian manifold $M$.
	Then,
	we say that $Z$ is positive
	if $U^{+}:=\{x\in M\mid \operatorname{sgn}(Z)>0\}$
	is nonempty.
\end{Definition}

This definition is intrinsic for a zonal flow.
Namely, we have the following.
\begin{Lemma}
\label{Lemma-gZZZ}
	Let 
	$Z=fX$ be a zonal flow on a Riemannian manifold $M$.
	Then,
	we have
	\begin{equation}
		g(\operatorname{grad}(\|Z\|^{2})
		+
		2\nabla_{Z}Z,
		2\nabla_{Z}Z)
		=
		-f^{2}\|X\|^{2}
		g(
		\operatorname{grad}(f^{2}),
		\operatorname{grad}(\|X\|^{2})).
		\label{eq-Lemma-gZZZ-1}
	\end{equation}
	Moreover,
	if $F$ is a function on
	$U_{0}:=\{x\in M\mid \operatorname{grad}(\|X\|^{2})\neq 0\}$
	satisfying
	\eqref{eq-Lemma-gradf^2=Fgrad},
	then we have
	\begin{equation}
		g(\operatorname{grad}(\|Z\|^{2})
		+
		2\nabla_{Z}Z,
		2\nabla_{Z}Z)
		=
		-Ff^{2}\|X\|^{2}
		\|\operatorname{grad}(\|X\|^{2})\|^{2}
		\label{eq-Lemma-gZZZ-2}
	\end{equation}
	on $U_{0}$.
	In particular,
	we have
	\begin{equation}
		\operatorname{sgn}(Z)
		=
		-
		\operatorname{sgn}(
		g(\operatorname{grad}(\|Z\|^{2})
		+
		2\nabla_{Z}Z,
		2\nabla_{Z}Z)
		)
	\end{equation}
	on $M$.
\end{Lemma}
\begin{proof}
By $Z=fX$
we have
	\begin{eqnarray*}
		\|Z\|^{2}
		&=&
		f^{2}\|X\|^{2},
		\\
		\operatorname{grad}(\|Z\|^{2})
		&=&
		\|X\|^{2}\operatorname{grad}(f^{2})
		+
		f^{2}\operatorname{grad}(\|X\|^{2}).
	\end{eqnarray*}
Moreover,
\eqref{eq-Lemma-nablaZZ} implies
\begin{eqnarray*}
	2\nabla_{Z}Z
	&=&
	-f^{2}\operatorname{grad}(\|X\|^{2}).
\end{eqnarray*}
These imply \eqref{eq-Lemma-gZZZ-1}.
Substituting \eqref{eq-Lemma-gradf^2=Fgrad},
we have \eqref{eq-Lemma-gZZZ-2}.
\end{proof}

\subsection{$S^{1}$-zonal flow}
In this section,
we introduce some class of zonal flows,
which is used in Section \ref{Section-sufficient-MC>0}
in order to
make a sufficient condition for the positivity of the
Misio{\l}ek curvature.
\begin{Definition}
\label{Def-zonal-S1}
	Let $Z=fX$ be a zonal flow on a Riemannian manifold $M$.
	Then,
	we say that $Z$ is a $S^{1}$-zonal flow
	if $X$ is induced by an $S^{1}$-action  on $M$.
\end{Definition}

Recall that $\operatorname{sgn}(Z)$
is defined in Definition \ref{Def-zonal-signature}.

\begin{Lemma}
\label{Lemma-U+-S1-inv-open}
	Let $Z$ be a positive $S^{1}$-zonal flow
	on a Riemannian manifold $M$.
	Then,
	$U^{+}:=\{x\in M \mid \operatorname{sgn}(Z)>0\}$ is 
	a $S^{1}$-invariant open subset of $M$.
\end{Lemma}
\begin{proof}
	Let $Z=fX$ with a function $f$ and a Killing vector field on $M$.
	Take the function $F$ on $U_{0}:=\{x\in M \mid \operatorname{grad}(\|X\|^{2})\neq 0\}$
	defined in Lemma \ref{Lemma-zonal-grad-collinear}.
	Then,
	we have
	$$
	X(F)=0
	$$
	by Lemma \ref{Lemma-zonal-grad-collinear}.
	Thus,
	$F$ is constant on any $S^{1}$-orbit
	because $X$ is induced by the $S^{1}$-action.
	This completes the proof
	because it is obvious $U^{+}$ is open.
\end{proof}

We recall the principal orbit type theorem
for compact Lie groups.
\begin{Fact}[{\cite[(1.1.4) and Lem.~1.1.5]{FU-prin}}]
\label{Fact-prin-orbit-type-thm}
Let a compact Lie group $G$ act on a connected compact manifold $M$.
Then,
there exist
a homogeneous $G$-space $G/H$
and
a dense open $G$-invariant subset $M_{pr}$ of $M$
such that
for any $x\in M_{pr}$,
there exists an $G$-equivariant open embedding
\begin{equation}
	\phi_{x} : G/H\times {\mathbb R}^{\dim M-\dim G/H}
	\to M
\end{equation}
satisfying
\begin{equation}
	\phi(G/H\times \{0\})=G\cdot x.
\end{equation}
Here, $G$ acts on 
$G/H\times {\mathbb R}^{\dim M-\dim G/H}$
via the first factor.
\end{Fact}

\begin{Corollary}
\label{Cor-prin-orbit-type-thm}
	Let $S^{1}$ act nontrivially on a connected compact manifold $M$.
Then,
there exist
a one-dimensional homogeneous $S^{1}$-space $N$
and
a dense open $S^{1}$-invariant subset $M_{pr}$ of $M$
such that
for any $x\in M_{pr}$,
there exists an $S^{1}$-equivariant open embedding
\begin{equation}
	\phi_{x} : N\times {\mathbb R}^{\dim M-\dim G/H}
	\to M
\end{equation}
satisfying
\begin{equation}
	\phi(N\times \{0\})=S^{1}\cdot x.
\end{equation}
\end{Corollary}
\begin{proof}
All assertion follows from Fact \ref{Fact-prin-orbit-type-thm}
except for the one-dimensionality of $N$.
By the nontriviality of the action,
	there exists a point $x_{0}\in M$
	such that the $S^{1}$-orbit of $x_{0}$
	is one-dimensional.
	By the continuity of the action,
	there exists an open neighborhood $W$ of $x_{0}$
	such that for any $x\in W$,
	the $S^{1}$-orbit of $x$
	is one-dimensional.
	This implies $N$ is one-dimensional
	because $M_{pr}\cap W$ is nonempty by the density of $M_{pr}$.
\end{proof}

The aim of this section
is to prove the following lemma,
which is used in Section \ref{Section-sufficient-MC>0}
in order to establish 
a sufficient
condition for the positivity of the
Micio{\l}ek curvature.

\begin{Lemma}
\label{Lemma-open-S1-emb}
	Let $Z$ be a positive $S^{1}$-zonal flow on a Riemannian manifold $M$.
	Then,
	there exist a one-dimensional homogeneous $S^{1}$-space $N$ and 
	a $S^{1}$-equivariant open embedding
	\begin{equation}
		\phi:N \times {\mathbb R}^{\dim M-1}
		\to
		M
	\end{equation}
	satisfying
	\begin{equation}
		\phi(N\times {\mathbb R}^{\dim M-1})\subset 
		U^{+}:=\{x\in M\mid \operatorname{sgn}(Z)>0\} .
	\end{equation}
\end{Lemma}

\begin{proof}
The assertion of this lemma is local,
we can assume $M$ is connected.
Moreover,
the positivity of $Z$ implies the $S^{1}$-action is nontrivial.
Thus,
by Corollary \ref{Cor-prin-orbit-type-thm},
there exist 
a homogeneous $S^{1}$-space $N$
and
a dense open subset $M_{pr}$ of $M$
	such that 
	for any $x\in M_{pr}$,
	there exists a $S^{1}$-equivariant open embedding
	\begin{eqnarray}
		\phi_{x}:N \times {\mathbb R}^{\dim M-1}
		&\to &
		M
		\label{eq-Lemma-phi-emb}
	\end{eqnarray}
	satisfying
\begin{equation}
	\phi(N\times \{0\})=S^{1}\cdot x.
\end{equation}
	
	On the other hand,
	$U^{+}=\{x\in M \mid \operatorname{sgn}(Z)>0\}$
	is nonempty because $Z$ is positive (see Definition \ref{Def-zonal-positive}).
	Because $M_{pr}$ 
	is dense by Corollary \ref{Cor-prin-orbit-type-thm}
	and $U^{+}$ is open by Lemma \ref{Lemma-U+-S1-inv-open},
	$U^{+}\cap M_{pr}$ is nonempty.
	Take $x_{0}\in U^{+}\cap M_{pr}$
	and consider the $S^{1}$-equivariant open embedding
	\begin{equation}
		\phi_{x_{0}}:N \times {\mathbb R}^{\dim M-1}
		\to 
		M.
	\end{equation}
	satisfying 
	\begin{equation}
	\phi(N\times \{0\})=S^{1}\cdot x_{0}
\end{equation}
	We note that
	$U^{+}$ is $S^{1}$-invariant open subset
	by Lemma \ref{Lemma-U+-S1-inv-open}.
	Thus,
	$x_{0}\in U^{+}$ implies
	\begin{equation}
		\phi_{x_{0}}(N \times \{0\})
		=
		S^{1}\cdot x_{0}\subset U^{+}
	\end{equation}
	Therefore,
	there exists 
    a sufficiently small open subset $W\subset {\mathbb R}^{\dim M-1}$,	
	we have
	\begin{equation}
		\phi_{x_{0}}(N\times W)
		\subset U^{+}
	\end{equation}
	because
	$U^{+}$ is $S^{1}$-invariant open subset
	and $\phi_{x_{0}}$ is $S^{1}$-equivariant.
	Taking some diffeomorphism $\psi:{\mathbb R}^{\dim M-1} \simeq W$
	and $\phi:=\phi_{x_{0}}\circ(\operatorname{id}_{N}\times \psi)$,
	we have the lemma.
\end{proof}

\section{Misio{\l}ek curvature of zonal flow}
\label{Section-Misiolek-zonal}
In this section,
we calculate the Misio{\l}ek curvature
for a zonal flow on an arbitrary compact Riemannian manifold.
Moreover,
we establish a sufficient condition for the
positivity of the Misio{\l}ek curvature.

\subsection{A formula of Misio{\l}ek curvature of zonal flow}
In this section,
we calculate of the Misio{\l}ek curvature
for a zonal flow
on an arbitrary compact Riemannian manifold $M$.

Set
\begin{eqnarray}
	\langle V,W\rangle &:=&
	\int_{M}g(V,W)\mu
	\\
	|V|^{2}
	&:=&
	\langle V,V\rangle
\end{eqnarray}
for any vector fields $V,W$ on $M$.
Note that we have
\begin{eqnarray}
	\langle
	fV,W
	\rangle
	=
	\langle
	V,fW
	\rangle
	\label{eq-fVW=VfW}
\end{eqnarray}
for any function $f$ on $M$.

Recall that the Misio{\l}ek curvature ${\mathfrak m}{\mathfrak c}$
is defined by
\begin{equation}
	{\mathfrak m}{\mathfrak c}_{Z,Y}
	:=
	-|[Z,Y]|^{2}
	-
	\langle
	Z,
	[[Z,Y],Y]
	\rangle
	\label{eq-def-MC-ZY}
	\end{equation}
for a stationary solution $Z$ of the Euler equation \eqref{E-eq-intro}
and a divergence-free vector field $Y$.
We note that
any zonal flow on $M$ 
is a stationary solution of \eqref{E-eq-intro}
by Proposition \ref{Prop-zonal-stationary}
(see Definition \ref{Def-zonal-flow-arbitrary} for the definition of a zonal flow).
\begin{Lemma}
\label{Lemma-zonal-MC-formula}
	Let $M$ be a compact Riemannian manifold 
	and $Z=fX$ be a zonal flow
	with a function $f$
	and a Killing vector field $X$ on $M$.
	Then,
	we have
	\begin{eqnarray}
		&&{\mathfrak m}{\mathfrak c}_{Z,Y}
		\label{eq-Lem-MC-fXY}
		\\
		&=&
		-
		|f[X,Y]|^{2}
		-
		\langle
		f^{2}X,
		[[X,Y],Y]
		\rangle
		+
		\langle
		2Y(f^{2})X,
		[X,Y]
		\rangle
		-
		\int_{M}
		Y^{2}(f^{2})
		\|X\|^{2}
		\mu.
		\nonumber
		\end{eqnarray}
\end{Lemma}
\begin{proof}
We have
\begin{eqnarray*}
	[Z,Y]
	&=&
	[fX,Y]
	\\
	&=&
	f[X,Y]
	-
	Y(f)
	X,
	\\
	~[[Z,Y],Y]
	&=&
	[f[X,Y]
	-
	Y(f)
	X,
	Y]
	\\
	&=&
	f[[X,Y],Y]
	-
	2Y(f)[X,Y]
	+
	Y^{2}(f)X.
\end{eqnarray*}
Therefore,
we have
\begin{eqnarray*}
	|[Z,Y]|^{2}
	&=&
	|f[X,Y]|^{2}
	-
	2
	\langle
	f[X,Y],
	Y(f)X
	\rangle
	+
	|Y(f)X|^{2}
	\end{eqnarray*}
	and
	\begin{eqnarray*}
		\langle
	Z,
	[[Z,Y],Y]
	\rangle
	&=&
	\langle
	fX,
	f[[X,Y],Y]
	\rangle
	-
	\langle
	fX,
	2Y(f)[X,Y]
	\rangle
	+
	\langle
	fX,
	Y^{2}(f)X
	\rangle
	\end{eqnarray*}
Therefore,
we have
\begin{eqnarray*}
	{\mathfrak m}{\mathfrak c}_{Z,Y}
	&=&
	-
	|f[X,Y]|^{2}
	+
	2
	\langle
	fY(f)[X,Y],
	X
	\rangle
	-
	\langle 
	Y(f)^{2}X,
	X
	\rangle
	\\
	&&
	-
	\langle
	f^{2}X,
	[[X,Y],Y]
	\rangle
	+
	2
	\langle
	fY(f)X,
	[X,Y]
	\rangle
	-
	\langle
	fY^{2}(f)X,
	X
	\rangle
\end{eqnarray*}
by
\eqref{eq-fVW=VfW}
and
\eqref{eq-def-MC-ZY}.
We note that
\begin{eqnarray}
	2fY(f)&=&Y(f^{2}),
	\\
	Y(f)^{2}+fY(f),
	&=&
	Y^{2}(f^{2}).
\end{eqnarray}
Then,
\begin{eqnarray*}
	{\mathfrak m}{\mathfrak c}_{Z,Y}
	&=&
	-
	|f[X,Y]|^{2}
	+
	2
	\langle
	Y(f^{2})[X,Y],
	X
	\rangle
	\\
	&&\quad
	-
	\langle
	f^{2}X,
	[[X,Y],Y]
	\rangle
	-
	\langle
	Y^{2}(f^{2})X,
	X
	\rangle.
\end{eqnarray*}
We note
\begin{eqnarray}
	\langle
	Y^{2}(f^{2})X,
	X
	\rangle
	&=&
	\int_{M}
	g(Y^{2}(f^{2})X,X)\mu
	\\
	&=&
	\int_{M}
	Y^{2}(f^{2})
	\|X\|^{2}\mu.
\end{eqnarray}
This completes the proof of 
\eqref{eq-Lem-MC-fXY}.
\end{proof}

Before stating the next lemma,
we recall that 
if $Z=fX$ is a zonal flow,
then,
there exists a function $F$ 
unique up to positive constant multiple
on 
$U_{0}:=\{x\in M\mid \operatorname{grad}(\|X\|^{2})\neq 0\}$
satisfying
\begin{equation}
\operatorname{grad}(f^{2})=F\operatorname{grad}(\|X\|^{2})
\label{eq-totyu-gradff=Fgrad}
\end{equation}
on $U_{0}$
by Lemma \ref{Lemma-zonal-grad-collinear}.
\begin{Lemma}
\label{Lemma-MC-commute-formula}
Let $Z=fX$ be a zonal flow on a compact Riemannian manifold $M$
and $F$ be a function on $U_{0}$
satisfying \eqref{eq-totyu-gradff=Fgrad}.
Then,
for any divergence-free vector field $Y$ on $M$
and
$[X,Y]=0$,
we have
\begin{eqnarray*}
	{\mathfrak m}{\mathfrak c}_{Z,Y}
	=
	\int_{U_{0}}FY(\|X\|^{2})^{2}\mu.
\end{eqnarray*}
\end{Lemma}

\begin{proof}
	By Lemma \ref{Lemma-zonal-MC-formula}
	and the assumption $[X,Y]=0$,
	we have
	\begin{eqnarray}
		{\mathfrak m}{\mathfrak c}_{Z,Y}=-\int_{M}Y^{2}(f^{2})\|X\|^{2}\mu.
	\end{eqnarray}
	Applying the Stokes theorem,
	we have
	\begin{eqnarray}
		{\mathfrak m}{\mathfrak c}_{Z,Y}
		&=&
		\int_{M}Y(f^{2})Y(\|X\|^{2})\mu
		\\
		&=&
		\int_{M}
		g(\operatorname{grad}(f^{2}),Y)Y(\|X\|^{2})\mu
		\\
		&=&
		\int_{U_{0}}
		g(\operatorname{grad}(f^{2}),Y)Y(\|X\|^{2})\mu
	\end{eqnarray}
	because
	\begin{eqnarray*}
		Y(\|X\|^{2})=
		g(\operatorname{grad}(\|X\|^{2}),Y)
		=0
	\end{eqnarray*}
	on $M\backslash U_{0}=\{x\in M\mid \operatorname{grad}(\|X\|^{2})=0\}$.
	Then,
	\eqref{eq-totyu-gradff=Fgrad}
	implies
	\begin{eqnarray}
		{\mathfrak m}{\mathfrak c}_{Z,Y}
		&=&
		\int_{U_{0}}
		Fg(\operatorname{grad}(\|X\|^{2}),Y)Y(\|X\|^{2})\mu
		\\
		&=&
		\int_{U_{0}}
		FY(\|X\|^{2})^{2}\mu.
	\end{eqnarray}
	This completes the proof.
\end{proof}

Recall that 
$\operatorname{sgn}(Z)$ is the signature of $F$
for a zonal flow $Z$,
see Definition \ref{Def-zonal-signature}
and Lemma \ref{Lemma-gZZZ}.

\begin{Corollary}
\label{Cor-zonal-MC>0}
	Let $Z$ be a zonal flow on a compact Riemannian manifold $M$
	and $Y$ a vector field on $M$.
	Suppose that
	$Y$ satisfies the following.
	\begin{enumerate}[(a)]
	\item $\operatorname{div}(Y)=0$,
	\item $[X,Y]=0$.
		\item $\operatorname{supp}(Y)\subset U^{+}:=\{x\in M\mid \operatorname{sgn}(Z)>0\}$.
		\item $Y(\|X\|^{2})\neq 0$ on $U^{+}$.
	\end{enumerate}
Then,
we have
\begin{eqnarray*}
	{\mathfrak m}{\mathfrak c}_{Z,Y}
	>0.
\end{eqnarray*}
\end{Corollary}
\begin{proof}
	Let $F$ be the function for $Z$
	satisfying \eqref{eq-totyu-gradff=Fgrad}
	on $U_{0}:=\{x\in M\mid \operatorname{grad}(\|X\|^{2})\neq 0\}$.
	Then,
	by the assumption on $Y$,
	Lemma \ref{Lemma-MC-commute-formula} implies
	\begin{eqnarray*}
		{\mathfrak m}{\mathfrak c}_{Z,Y}
	&=&
	\int_{U_{0}}FY(\|X\|^{2})^{2}\mu
	\\
	&=&
	\int_{U^{+}}FY(\|X\|^{2})^{2}\mu.
	\end{eqnarray*}
	Then,
	because
	$F$ is positive on $U^{+}$ by definition
	and 
	$Y(\|X\|^{2})\neq 0$ by the assumption,
	we have the corollary.
\end{proof}

\subsection{A sufficient condition for ${\mathfrak m}{\mathfrak c}>0$}
\label{Section-sufficient-MC>0}
In this section,
we give a sufficient condition for ${\mathfrak m}{\mathfrak c}_{Z,Y}>0$.
Namely,
we prove Theorem \ref{Theorem-Main}.

\begin{proof}[Proof of Theorem \ref{Theorem-Main}]
By Corollary \ref{Cor-zonal-MC>0},
it is sufficient to show that
there exists a vector field $Y$ on $M$,
which satisfies the following conditions:
\begin{enumerate}[(a)]
\item $\operatorname{div}(Y)=0$,
	\item $[X,Y]=0$,
		\item $\operatorname{supp}(Y)\subset U^{+}:=\{x\in M\mid \operatorname{sgn}(Z)>0\}$,
		\item $Y(\|X\|^{2})\neq 0$ on $U^{+}$.
	\end{enumerate}

On the other hand,
by Lemma \ref{Lemma-open-S1-emb},
there exist a one-dimensional homogeneous $S^{1}$-space $N$ and 
	a $S^{1}$-equivariant open embedding
	\begin{equation}
		\phi:N \times {\mathbb R}^{\dim M-1}
		\to
		M
	\end{equation}
	satisfying
	\begin{equation}
		\phi(N\times {\mathbb R}^{\dim M-1})\subset 
		U^{+}:=\{x\in M\mid \operatorname{sgn}(Z)>0\}.
	\end{equation}
By this embedding, we regard 
\begin{equation}
	N\times {\mathbb R}^{\dim M-1} \subset U^{+}.
\end{equation}
Thus,
this proposition 
follows from 
Lemma \ref{Lemma-existence-Y-nonzero}
given below
by substituting $h=\|X\|^{2}$
and $\omega=\mu$.
We note that 
Lemma \ref{Lemma-existence-Y-nonzero} is applicable 
because $X(\|X\|^{2})=0$ by Lemma \ref{Lemma-Killing-XXX=0} 
and
$\|X\|^{2}$ is non-constant by the non-geodesic assumption of $Z$
(see also Corollary \ref{Cor-geodesic-equiv}).
\end{proof}

\begin{Lemma}
\label{Lemma-existence-Y-nonzero}
	Let $N$ be a a one-dimensional homogeneous $S^{1}$-space,
	$\omega$ be a $S^{1}$-invariant volume form on $N\times {\mathbb R}^{k}$,
	$X$ a vector field on $N\times {\mathbb R}^{k}$
	induced by the $S^{1}$-action,
	and $h$ a non-constant function on $N\times {\mathbb R}^{k}$
	satisfying $X(h)=0$.
	Then, if $k\geq 2$, there exists
	a compactly supported divergence-free vector field $Y$ on $N\times {\mathbb R}^{k}$ satisfying $[X,Y]=0$
	such that $Y(h)\neq 0$.
\end{Lemma}

\begin{proof}
Although this is probably obvious,
	we prove this lemma in Appendix \ref{Appendix-proof-of-Lemma-Y-existence}
	for the completeness.
\end{proof}

\section{2D ellipsoid case}
\label{Section-2D-ellipsoid}
In this section,
we show that Definition \ref{Def-zonal-flow-arbitrary}
is equivalent to the definition of \cite[(4.3)]{TY1}
in the case that $M$ is a two-dimensional ellipsoid.

\subsection{Setting}
In this section,
we calculate some formulae on two-dimensional
ellipsoid.

Let
$$
M_{a}^{2}
:=
\{
(x,y,z)\in{\mathbb R}^{3}
\mid
x^{2}+y^{2}
=
a^{2}(1-z^{2})
\}.
$$
Take the spherical coordinate (\cite[Sect.~4]{TY1})
\begin{equation*}
\begin{array}{cccccc}
\phi
&
:
&
(-d,d)\times (-\pi,\pi)
&
\to 
&
M^{2}_{a}
&
\\
&&
(r,\theta)
&
\mapsto
&
(c_{1}(r)\cos\theta,
c_{1}(r)\sin\theta,
c_{2}(r)
).
\end{array}
\end{equation*}
Note that
$c(r)$ satisfies 
$\dot{c}_{1}(r)^{2}+
\dot{c}_{2}(r)^{2}=1$
for any $r\in (-d,d)$,
namely,
$c(r)$ is parameterized by arc length
and $d$ is some suitable positive constant.
Moreover,
$c_{1}(r)$ is a non-constant function.

Let $g$ be a Riemannian metric on $M^{2}_{a}$ induced by ${\mathbb R}^{3}$.
Then,
we have
\begin{eqnarray}
	g=(g_{ij})=(g(\partial_{i},\partial_{j}))=
	\begin{pmatrix}
		1&0\\
		0&c_{1}^{2}
	\end{pmatrix},
	\label{eq-gij-2D-ellipsoid}
\end{eqnarray}
where we make a coorespondence
\begin{equation}
	1 \mapsto r,\qquad
	2\mapsto \theta,\end{equation}
Then, 
the inverse matrix of $g$ is
\begin{equation}
	g^{-1}=(g^{ij})=
	\begin{pmatrix}
		1&0\\
		0&c_{1}^{-2}
	\end{pmatrix}.
\end{equation}
and
\begin{eqnarray}
	\operatorname{grad}p&=&
	\partial_{r}p\partial_{r}+c_{1}^{-2}\partial_{\theta}p\partial_{\theta}
	\label{eq-grad-on-2D-ellipsoid}
\end{eqnarray}	
for any function $p$ on $M^{2}_{a}$.
Recall that the Christoffel symbols are given by
\begin{equation}
	\Gamma^{k}_{ij}
	=
	\frac{1}{2}
	\left(
	\sum_{a=1}^{2}g^{ka}
	\left(
	\partial_{i}g_{ja}
	+
	\partial_{j}g_{ia}
	-
	\partial_{a}g_{ij}
	\right)
	\right).
\end{equation}
Thus, we have
\begin{equation}
\Gamma^{1}_{22}=-c_{1}\partial_r c_{1}, \quad
\Gamma^{2}_{12}=
\Gamma^{2}_{21}=\frac{\partial_r c_{1}}{c_{1}}.
\label{eq-Christoffel-2D-ellipsoid}
\end{equation}
Thus,
we have
\begin{eqnarray}
(\nabla_{\partial_{i}}\partial_{j})
&=&
	\begin{pmatrix}
		\nabla_{\partial_{1}}\partial_{1}&
		\nabla_{\partial_{1}}\partial_{2}&
\\
		\nabla_{\partial_{2}}\partial_{1}&
		\nabla_{\partial_{2}}\partial_{2}&
	\end{pmatrix}
	\label{eq-nabla-ij-on-2D-ellipsoid}
	\\
	&=&
	\begin{pmatrix}
		0&
		\frac{\partial_r c_{1}}{c_{1}}\partial_{2}
		\\
		\frac{\partial_r c_{1}}{c_{1}}\partial_{2}&
		-c_{1}\partial_r c_{1}
		\partial_{1}
	\end{pmatrix}.
	\nonumber
\end{eqnarray}

\subsection{Killing vector fields on $M^{2}_{a}$}
In this section,
we classify Killing vector fields on $M^{2}_{a}$.

\begin{Lemma}
\label{Lemma-Killing-on-2D-ellipsoid}
	The space of Killing vector fields on $M^{2}_{a}$
	is spanned by 
	\begin{equation}
		\partial_{\theta}.
	\end{equation}
\end{Lemma}
\begin{proof}
By \eqref{eq-Killing-Def-VW},
a vector field $X$ on $M^{2}_{a}$
is Killing if and only if
\begin{equation}
	g(\nabla_{\partial_{i}}X,\partial_{j})
	=
	-
	g(\nabla_{\partial_{j}}X,\partial_{i})
\end{equation}
holds for any $i,j\in\{1,2\}$.
From this,
we have the lemma 
by easy but lengthy computation.
\end{proof}
\begin{Lemma}
\label{Lemma-Killing-S1-on-2D-ellipsoid}
	Any Killing vector field on $M^{2}_{a}$ 
	is induced by a $S^{1}$-action.
\end{Lemma}
\begin{proof}
	This is obvious Lemma \ref{Lemma-Killing-on-2D-ellipsoid}.
	In fact,
	$\partial_{\theta}$ is induced by
	a $S^{1}$-action
	\begin{equation}
		e^{it}\cdot (r,\theta):=(r,\theta+t),
	\end{equation}
	where $e^{it}\in \{z\in{\mathbb C}\mid |z|=1\}\simeq S^{1}$.
\end{proof}
\begin{Lemma}
\label{Lemma-Killing-non-geodesic-on-2D-ellipsoid}
	Any Killing vector field on $M^{2}_{a}$ is non-geodesic.
\end{Lemma}
\begin{proof}
	We have
	$$
	\|X\|^{2}=c_{1}^{2}
	$$
	by \eqref{eq-gij-2D-ellipsoid}.
	We note that $c_{1}$ is non-constant.
	Thus,
	Corollary \ref{Cor-geodesic-equiv} implies the lemma.
\end{proof}

\subsection{Zonal flows on $M^{2}_{a}$}
In this section,
we classify zonal flows on $M^{2}_{a}$.

\begin{Lemma}
\label{Lemma-classify-zonal-on-2D-ellipsoid}
Let $Z$ be a zonal flow in the sense of Definition \ref{Def-zonal-flow-arbitrary} on $M^{2}_{a}$.
Then,
there exists a one-variable function $f=f(r)$
such that
\begin{equation}
	Z=f(r)\partial_{\theta}.
\end{equation}	
\end{Lemma}
\begin{proof}
	By Lemma \ref{Lemma-Killing-on-2D-ellipsoid}
	and Definition \ref{Def-zonal-flow-arbitrary},
	there exists a function $f(r,\theta)$ on $M^{2}_{a}$
	such that
	\begin{equation}
		Z=f(r,\theta)\partial_{\theta}.
	\end{equation}
	However,
	$\operatorname{div}(Z)=0$ and Lemma \ref{Lemma-zonal-div=0} imply
	\begin{equation}
		\partial_{\theta}f=0.
	\end{equation}
	Thus,
	$f=f(r)$ only depends on the variable $r$.
	Moreover,
	we have
	\begin{equation}
		\nabla_{Z}Z=f^{2}\nabla_{\partial_{\theta}}\partial_{\theta}
		=
		-
		f^{2}c_{1}\partial_r c_{1}
		\partial_{2}
	\end{equation}
	by \eqref{eq-nabla-ij-on-2D-ellipsoid}.
	Let $p$ be a primitive function of $f^{2}c_{1}\partial_r c_{1}$,
	namely we define $p$ by 
	$$
	\partial_{r}p
	=
	f^{2}c_{1}\partial_r c_{1}.
	$$
	Then,
	we have
	\begin{equation}
		\nabla_{Z}Z=-\operatorname{grad}p
	\end{equation}
	by \eqref{eq-grad-on-2D-ellipsoid}.
	This completes the lemma.
\end{proof}
\begin{Remark}
	Lemma \ref{Lemma-classify-zonal-on-2D-ellipsoid}
	implies that
	Definition \ref{Def-zonal-flow-arbitrary}
	is 
	equivalent to 
	\cite[(4.3)]{TY1}.
\end{Remark}

\begin{Corollary}
\label{Cor-any-zonal-is-S1-geodesic-on-2D-ellipsoid}
	Any zonal flow on $M^{2}_{a}$
	is a non-geodesic $S^{1}$-zonal flow.
\end{Corollary}
\begin{proof}
	This follows from
	Lemmas \ref{Lemma-Killing-S1-on-2D-ellipsoid}
	and
	\ref{Lemma-Killing-non-geodesic-on-2D-ellipsoid}.
\end{proof}

\begin{Remark}
\label{Remark-S2-zonal}
Definition \ref{Def-zonal-flow-arbitrary}
is essentially equivalent to that of \cite{TY1}
in the case $M=S^{2}$.
In precisely,
we have the following.
Let
$$
S^{2}:=\{(x,y,z)\in{\mathbb R}^{3}\mid x^{2}+y^{2}+z^{2}=1\}
$$
\begin{Lemma}
	Any zonal flow in the sense of Definition \ref{Def-zonal-flow-arbitrary} on $S^{2}$ can be transformed by an isometry to the form
	\begin{equation}
		Z=f(z)(y\partial_{x}-x\partial_{y}),
	\end{equation}
	where $f(z)$ is a one-variable function.
\end{Lemma}
The proof is the same for Lemma \ref{Lemma-classify-zonal-on-2D-ellipsoid}
using the following lemma instead of Lemma \ref{Lemma-Killing-on-2D-ellipsoid}.
\begin{Lemma}
\label{Lemma-S2-Killing-trans}
	Any Killing flow on $S^{2}$ can be transformed by an isometry
	to $\partial_{\theta}$.
\end{Lemma}
\begin{proof}
We prove this lemma in Appendix \ref{Appendix-proof-Lemma-S2-Killing}
by using the theory of compact Lie groups.
\end{proof}
\end{Remark}

\section{Classification of zonal flows on 3D ellipsoids}
\label{Section-class-zonal-3D-ellip}
In this section,
we classify 
zonal flows on
three-dimensional ellipsoids.

\subsection{Setting}
In this section,
we calculate some formulae on three-dimensional
ellipsoid.

Let 
\begin{equation}
	M_{a}^{3}:=\{(x,y,z,w)\in{\mathbb R}^{4}\mid x^{2}+y^{2}=a^{2}(1-z^{2}-w^{2})\}.
\end{equation}
Take a coordinate
\begin{eqnarray}
\begin{array}{ccc}
(-\pi,\pi)\times (-\pi,\pi)\times (0,\frac{\pi}{2})
& \to &
S^{3}\\
\rotatebox{90}{$\in$}
&&
\rotatebox{90}{$\in$}\\
	(\xi,\mu,\chi)
	&\mapsto &
	(a\cos(\xi)\sin(\chi),a\sin(\xi)\sin(\chi),\cos(\mu)\cos(\chi),\sin(\mu)\cos(\chi)).
	\end{array}
	\label{Hopf}
\end{eqnarray}
we have
\begin{eqnarray}
	\partial_{\xi}
	&\mapsto &
	(-a\sin(\xi)\sin(\chi),a\cos(\xi)\sin(\chi),0,0),
	\nonumber\\
	\partial_{\mu}
	&\mapsto &
	(0,0,-\sin(\mu)\cos(\chi),\cos(\mu)\cos(\chi)),
	\nonumber\\
	\partial_{\chi}
	&\mapsto &
	(a\cos(\xi)\cos(\chi),a\sin(\xi)\cos(\chi),-\cos(\mu)\sin(\chi),-\sin(\mu)\sin(\chi)).
	\nonumber
\end{eqnarray}
Let $g$ be a Riemannian metric on $S^{3}$ induced by ${\mathbb R}^{4}$.
Then,
we have
\begin{eqnarray}
	g=(g_{ij})=(g(\partial_{i},\partial_{j}))=
	\begin{pmatrix}
		a^{2}\sin^{2}(\chi)&0&0\\
		0&\cos^{2}(\chi)&0\\
		0&0&a^{2}\cos^{2}(\chi)+\sin^{2}(\chi)
	\end{pmatrix},
	\label{eq-gij-3D-ellipsoid}
\end{eqnarray}
where we make a coorespondence
\begin{equation}
	1 \mapsto \xi,\qquad
	2\mapsto \mu,\qquad
	3\mapsto \chi.
\end{equation}
Then, 
the inverse matrix of $g$ is
\begin{equation}
	g^{-1}=(g^{ij})=
	\begin{pmatrix}
		1/a^{2}\sin^{2}(\chi)&0&0\\
		0&1/\cos^{2}(\chi)&0\\
		0&0&1/(a^{2}\cos^{2}(\chi)+\sin^{2}(\chi))
	\end{pmatrix}.
\end{equation}
and
\begin{eqnarray}
	\operatorname{grad}f&=&
	\frac{\partial_{1}f}{a^{2}\sin^{2}(\chi)}
	\partial_{1}
	+
	\frac{\partial_{2}f}{\cos^{2}(\chi)}
	\partial_{2}
		+
	\frac{\partial_{3}f}{a^{2}\cos^{2}(\chi)+\sin^{2}(\chi)}
	\partial_{3}.
	\label{eq-grad-3D-ellipsoid-Hopf}
\end{eqnarray}	
for any function $f$ on $M^{3}_{a}$.

Moreover, we have
\begin{eqnarray}
	\Gamma^{1}_{13}=\Gamma^{1}_{31}=\frac{1}{\tan(\chi)},&&
	\Gamma^{2}_{23}=\Gamma^{2}_{32}=-\tan(\chi),
	\\
	\Gamma^{3}_{11}=\frac{-a^{2}\sin(\chi)\cos(\chi)}{a^{2}\cos^{2}(\chi)+\sin^{2}(\chi)},&&
	\Gamma^{3}_{22}=\frac{\sin(\chi)\cos(\chi)}{a^{2}\cos^{2}(\chi)+\sin^{2}(\chi)},
	\\
	\Gamma^{3}_{33}=\frac{(1-a^{2})\sin(\chi)\cos(\chi)}{a^{2}\cos^{2}(\chi)+\sin^{2}(\chi)}.
\end{eqnarray}
Thus,
we have
\begin{eqnarray}
(\nabla_{\partial_{i}}\partial_{j})
&=&
	\begin{pmatrix}
		\nabla_{\partial_{1}}\partial_{1}&
		\nabla_{\partial_{1}}\partial_{2}&
		\nabla_{\partial_{1}}\partial_{3}\\
		\nabla_{\partial_{2}}\partial_{1}&
		\nabla_{\partial_{2}}\partial_{2}&
		\nabla_{\partial_{2}}\partial_{3}\\
		\nabla_{\partial_{3}}\partial_{1}&
		\nabla_{\partial_{3}}\partial_{2}&
		\nabla_{\partial_{3}}\partial_{3}
	\end{pmatrix}
	\label{nabla-ij}
	\\
	&=&
	\begin{pmatrix}
		\frac{-a^{2}\sin(\chi)\cos(\chi)}{a^{2}\cos^{2}(\chi)+\sin^{2}(\chi)}
		\partial_{3}&
		0&
		1/\tan(\chi)\partial_{1}\\
		0&
		\frac{\sin(\chi)\cos(\chi)}{a^{2}\cos^{2}(\chi)+\sin^{2}(\chi)}
		\partial_{3}&
		-\tan(\chi)\partial_{2}\\
		1/\tan(\chi)\partial_{1}&
		-\tan(\chi)\partial_{2}&
		\frac{(1-a^{2})\sin(\chi)\cos(\chi)}{a^{2}\cos^{2}(\chi)+\sin^{2}(\chi)}
		\partial_{3}
	\end{pmatrix}.
	\nonumber
\end{eqnarray}

\subsection{Classification of Killing vector fields on $M^{3}_{a}$}
In this section,
we classify Killing vector fields on $M^{3}_{a}$
with $a\neq 1$.

\begin{Lemma}
\label{Lemma-Classification-3D-ellipsoid}
Let 
$a\neq 1$.
Then,
the space of Killing vector fields on $M^{3}_{a}$
is spanned by
two vector fields
\begin{equation}
	\partial_{\xi},
	\quad
	\partial_{\mu}.
\end{equation}
\end{Lemma}

\begin{proof}
	Because the proof is similar to Lemma \ref{Lemma-Killing-on-2D-ellipsoid},
	we omit it.
\end{proof}

\begin{Lemma}
\label{Lemma-Killing-geodesic-3D-ellipsoid}
	Let $a\neq 1$
	and $X$ a Killing vector field.
	Then,
	there exists $p,q\in{\mathbb R}$
	such that
	\begin{equation}
	X=p\partial_{1}+q\partial_{2}.
	\label{eq-Lemma-Killing-X=pq}
	\end{equation}
	Moreover
	$\|X\|^{2}$ is constant if and only if 
	$$
	a^{2}p^{2}= q^{2}.
	$$
\end{Lemma}
\begin{proof}
\eqref{eq-Lemma-Killing-X=pq} 
is immediate corollary of Lemma \ref{Lemma-Classification-3D-ellipsoid}.
By \eqref{eq-gij-3D-ellipsoid},
	we have
	\begin{eqnarray*}
		\|X\|^{2}
		&=&
		a^{2}p^{2}\sin^{2}(\chi)
		+
		q^{2}\cos^{2}(\chi)
		\\
		&=&
		a^{2}p^{2}
		-
		(a^{2}p^{2}-q^{2})
		\cos^{2}(\chi).
	\end{eqnarray*}
	This completes the proof.	
\end{proof}

\subsection{Classification of zonal flows on $M^{3}_{a}$}
In this section,
we classify zonal flows on $M^{3}_{a}$.
First,
we classify geodesic zonal flows on $M^{3}_{a}$
(see Definition \ref{Def-zonal-geodesic-arbitrary}).
\begin{Proposition}
\label{Prop-zonal-geodesic-on-3D-ellipsoid}
	Let $a\neq 1$
	and $Z$ be a geodesic zonal flow.
	Then,
	there exist
	$p,q\in{\mathbb R}$
	and a function $f$ on $M^{3}_{a}$
	satisfying
	$a^{2}p^{2}=q^{2}$
	and
	$(p\partial_{1}+q\partial_{2})(f)=0$
	such that
	\begin{equation}
		Z=f(p\partial_{1}+q\partial_{2}).
	\end{equation}
\end{Proposition}
\begin{proof}
	Because 
	$Z$ is a zonal flow,
	there exists
	a function $f=f(\xi,\mu,\chi)$ 
	and Killing vector field $X$ on $M^{3}_{a}$
	satisfying
	\begin{equation}
	Z=fX.
	\end{equation}
	Lemma \ref{Lemma-Killing-geodesic-3D-ellipsoid} implies
	there exist
	$p,q\in{\mathbb R}$
	satisfing $a^{2}p^{2}=q^{2}$
	such that
	\begin{equation}
		X=p\partial_{1}+q\partial_{2}.
	\end{equation}
	This completes the proof
	by Lemma \ref{Lemma-geodesic-zonal-fZ}.
\end{proof}

Second,
we classify non-geodesic zonal flows on $M^{3}_{a}$.

\begin{Proposition}
\label{Prop-zonal-non-geodesic-on-3D-ellipsoid}
	Let $a\neq 1$
	and 
	$Z$ be a non-geodesic zonal flow on $M^{3}_{a}$.
	Then,
	there exist $p,q\in{\mathbb R}$
	satisfying $a^{2}p^{2}\neq q^{2}$
	and a one-variable function $f=f(\chi)$
	such that
	\begin{equation}
		Z=f(p\partial_{1}+q\partial_{2}).
	\end{equation}
\end{Proposition}
\begin{proof}
	Because 
	$Z$ is a zonal flow,
	there exists
	a function $f=f(\xi,\mu,\chi)$ 
	and Killing vector field $X$ on $M^{3}_{a}$
	satisfying
	\begin{equation}
	Z=fX.
	\end{equation}
	Lemma \ref{Lemma-Killing-geodesic-3D-ellipsoid} implies
	there exist
	$p,q\in{\mathbb R}$
	with $a^{2}p^{2}\neq q^{2}$
	such that
	\begin{equation}
		X=p\partial_{1}+q\partial_{2}.
	\end{equation}
	By \eqref{eq-gij-3D-ellipsoid} and \eqref{eq-grad-3D-ellipsoid-Hopf},
	we have
	\begin{eqnarray}
		\operatorname{grad}(\|X\|^{2})
		&=&
		\frac{2(a^{2}p^{2}-q^{2})\sin(\chi)\cos(\chi)}{a^{2}\cos^{2}(\chi)+\sin^{2}(\chi)}
		\partial_{3}
		\neq 0.
		\label{eq-grad(X^2)-2aapp}
	\end{eqnarray}
	Thus,
	$f$ satisfies
	\begin{equation}
		\partial_{1}f
		=
		\partial_{2}f
		=
		0
	\end{equation}
	by Lemma \ref{Lemma-zonal-grad-collinear}, \eqref{eq-grad-3D-ellipsoid-Hopf} and \eqref{eq-grad(X^2)-2aapp}.
\end{proof}

At the end of this section,
we classify $S^{1}$-zonal flows on $M^{3}_{a}$.
We set $\frac{q}{p}=\infty$
where $p,q\in{\mathbb R}$
if $p=0$.
\begin{Proposition}
\label{Proposition-S1-zonal-on-3D-ellipsoid}
	Let
	$Z=f(p\partial_{1}+q\partial_{2})$
	be a zonal flow on $M^{3}_{a}$
	with $p,q\in{\mathbb R}$
	and a function $f$.
	Then,
	$Z$ is a $S^{1}$-zonal flow 
	if and only if
	$\frac{q}{p}\in{\mathbb Q}\cup\{\infty\}$.
\end{Proposition}
\begin{proof}
	Let $\chi_{0}\in(0,\frac{\pi}{2})$.
	Then,
	the image by the coordinate 
	\eqref{Hopf}
	of
	\begin{eqnarray*}
		\left\{(\xi,\mu,\chi)
		\in (-\pi,\pi)\times (-\pi,\pi)\times (0,\frac{\pi}{2})
		\middle|
		\chi=\chi_{0}
		\right\}
	\end{eqnarray*}
	is 
	isometric to
	the two-dimensional 
	flat torus
	\begin{eqnarray*}
		\{(x,y)\in {\mathbb R}/2\pi{\mathbb Z}\times {\mathbb R}/2\pi{\mathbb Z}\mid (x,y)\in {\mathbb R}^{2}\}.
	\end{eqnarray*}
	Moreover,
	this isometry transforms 
	$\partial_{1}$ and $\partial_{2}$
	to
	$\partial_{x}$ and $\partial_{y}$,
	respectively.
	Then,
	the proposition is obvious.
\end{proof}

\begin{Corollary}
\label{Cor-non-geodesic-S1-zonal-on-3D-ellipsoid}
	Let $Z$ be a zonal flow on $M^{3}_{a}$.
	Then,
	$Z$ is a non-geodesic $S^{1}$-zonal flow
	if and only if 
	there exist 
	$p,q\in{\mathbb R}$
	satisfying
	$a^{2}p^{2}\neq q^{2}$
	and
	$\frac{q}{p}\in{\mathbb Q}\cup\{\infty\}$
	and 
	a one-variable function $f(\chi)$
	such that
	\begin{eqnarray*}
		Z=f
		(p\partial_{1}+q\partial_{2}).
	\end{eqnarray*}
\end{Corollary}
\begin{proof}
	This is obvious by
	Propositions
	\ref{Prop-zonal-non-geodesic-on-3D-ellipsoid}
	and
	\ref{Proposition-S1-zonal-on-3D-ellipsoid}.
\end{proof}

\begin{Remark}
	We can also classify zonal flows in the case $a=1$,
	namely, in the case $M^{3}_{a}=S^{3}$.
	In precisely,
	we have the following.
	\begin{Proposition}
	\label{Prop-zonal-geodesic-on-3D-sphere}
		Let $a= 1$
	and $Z$ be a geodesic zonal flow.
	Then,
	there exist
	$p,q\in{\mathbb R}$
	and a function $f=f(\xi,\mu,\chi)$ on $S^{3}$
	satisfying
	$a^{2}p^{2}=q^{2}$
	and
	$(p\partial_{1}+q\partial_{2})(f)=0$
	such that
	$Z$ is transformed by an isometry to
	\begin{equation}
		f(p\partial_{1}+q\partial_{2}).
	\end{equation}
	\end{Proposition}
	
	\begin{Proposition}
	\label{Prop-zonal-non-geodesic-on-3D-sphere}
	Let $a=1$
	and 
	$Z$ be a non-geodesic zonal flow on $S^{3}$.
	Then,
	there exist $p,q\in{\mathbb R}$
	satisfying $a^{2}p^{2}\neq q^{2}$
	and a one-variable function $f=f(\chi)$
	such that
	$Z$ is transformed by an isometry to
	\begin{equation}
		f(p\partial_{1}+q\partial_{2}).
	\end{equation}
	\end{Proposition}
	The proofs are the same 
	for Propositions \ref{Prop-zonal-geodesic-on-3D-ellipsoid}
	and \ref{Prop-zonal-non-geodesic-on-3D-ellipsoid}
using the following lemma instead of Lemma \ref{Lemma-Classification-3D-ellipsoid}.

\begin{Lemma}
	Any Killing flow on $S^{3}$ can be transformed by an isometry
	to 
	$$
		p\partial_{\xi},
	+
	q\partial_{\mu}
	$$
	with $p,q\in{\mathbb R}$.
\end{Lemma}
\begin{proof}
	Because the proof is similar to Lemma \ref{Lemma-S2-Killing-trans},
	we omit it.
\end{proof}

\end{Remark}

\section{Misio{\l}ek curvature on $M^{3}_{a}$}
\label{Section-prove-Theorem-Main-2}
Finally,
we prove Theorem \ref{Theorem-Main-2}
in this section.
Recall that the positivity of zonal flow is defined in 
Definition \ref{Def-zonal-positive}.
\begin{proof}[Proof of Theorem \ref{Theorem-Main-2}]
By Theorem \ref{Theorem-Main},
it is enough to show that
any zonal flow on $M^{3}_{a}$
with $\operatorname{supp}(Z)\subset M^{3}_{a}\backslash 
(N\cup S)$
is positive.
By Corollary \ref{Cor-non-geodesic-S1-zonal-on-3D-ellipsoid},
we have
\begin{equation}
Z=f(p\partial_{1}+q\partial_{2}),
\end{equation}
where $f=f(\chi)$ is a one-variable function and
$p,q\in{\mathbb R}$
satisfying $a^{2}p^{2}\neq q^{2}$.
By the assumption
$\operatorname{supp}(Z)\subset M^{3}_{a}\backslash 
(N\cup S)$,
there exists $\varepsilon>0$
such that
\begin{eqnarray*}
	f(\chi)\equiv0,\quad 
	\text{
	on
	}
	\chi \in 
	[0,\varepsilon)
	\cup
	\left(\frac{\pi}{2}-\varepsilon,\frac{\pi}{2}\right],
\end{eqnarray*}
which implies the existence of a maximal value of $f^{2}$
on $\chi\in(0,\frac{\pi}{2})$.

On the other hand,
we have
	\begin{eqnarray}
		\operatorname{grad}(\|X\|^{2})
		&=&
		\frac{2(a^{2}p^{2}-q^{2})\sin(\chi)\cos(\chi)}{a^{2}\cos^{2}(\chi)+\sin^{2}(\chi)}
		\partial_{3}
	\end{eqnarray}
	by \eqref{eq-gij-3D-ellipsoid} and \eqref{eq-grad-3D-ellipsoid-Hopf}.
	Thus, $\operatorname{grad}(\|X\|^{2})$ is always
	positive or negative 
	on $\chi\in(0,\frac{\pi}{2})$
	by $a^{2}p^{2}\neq q^{2}$.
	Moreover,
	we have
	\begin{eqnarray*}
		\operatorname{grad}(f^{2})
		=
		\frac{\partial_{3}f^{2}}{a^{2}\cos^{2}(\chi)+\sin^{2}(\chi)}
	\partial_{3}
	\end{eqnarray*}
	by
	\eqref{eq-grad-3D-ellipsoid-Hopf}.
	However,
	the existence of 
	a maximal value of $f^{2}$
	implies
	that
	$\partial_{\chi}(f^{2})$
	takes both 
	positive and negative values
	on $\chi\in(0,\frac{\pi}{2})$.
	Because 
	$a^{2}\cos^{2}(\chi)+\sin^{2}(\chi)$
	is always positive,
	this completes the proof.
\end{proof}

\appendix

\section{Proof of Lemma \ref{Lemma-S2-Killing-trans}}
\label{Appendix-proof-Lemma-S2-Killing}
In this section,
we prove Lemma \ref{Lemma-S2-Killing-trans}.
For this end,
we recall the following elementary fact of the theory of compact groups.
\begin{Fact}[{\cite[Prop.~(3.8.1)]{Duistermaat}}]
	Let $K$ be a compact Lie group 
	with Lie algebra ${\mathfrak k}$
	and
	${\mathfrak t}$
	a maximal abelian subalgebra of
	${\mathfrak k}$
	Then,
	we have
	\begin{eqnarray*}
	K\cdot {\mathfrak t}
	={\mathfrak k}.
	\end{eqnarray*}
\end{Fact}
\begin{proof}[Proof of Lemma \ref{Lemma-S2-Killing-trans}]
Let $SO(3)$ be the orthogonal group.
Then, $SO(3)$ acts on ${\mathbb R}^{3}$
by isometries.
Thus,
we have
$$SO(3)/SO(2)\simeq S^{2}.$$
Moreover,
the Lie algebra ${\mathfrak g}$ of $G:=SO(3)$
is isomorphic to the space of Killing vector fields on $S^{2}$.
We note that any maximal abelian subalgebra of ${\mathfrak g}$
is one-dimensional.
Therefore,
any Killing vector fields on $S^{2}$
can be transformed by an isometry 
to $\partial_{\theta}$.
This completes the proof.
\end{proof}

\section{Proof of Lemma \ref{Lemma-existence-Y-nonzero}}
\label{Appendix-proof-of-Lemma-Y-existence}
In this section,
we prove Lemma \ref{Lemma-existence-Y-nonzero}.
For this end,
we prepare some elementary lemmas
in the following three sections.

\subsection{Compactly supported divergence free vector fields on ${\mathbb R}^{k}$}
In this section,
we prove a elementary lemma,
which states that 
there exists many compactly supported divergence-free vector fields on ${\mathbb R}^{k}$
with respect to any volume form on ${\mathbb R}^{k}$.

Let $M$ be an orientable manifold
with a volume form $\mu$,
namely,
$\mu$ is a nowhere vanishing $(\dim M)$-form on $M$.
Note that we do not assume $M$ is a Riemannian manifold.
Recall that 
the divergence of a vector field $Y$ on $M$ 
is defined by
\begin{equation}
	\operatorname{div}(Y)\mu
	=
	L_{Y}(\mu)
	=
	d\circ i_{Y}(\mu),
	\label{eq-def-div-di}
\end{equation}
where $L_{Y}$ is the Lie derivative,
$i_{Y}$ is the interior derivative
and
$d$ is the exterior derivative.
We also write $\operatorname{div}_{\mu}(Y):=
\operatorname{div}(Y)$
when we want to emphasis that
we calculate the divergence with respect to $\mu$.
Let 
\begin{eqnarray}
	{\mathfrak X}_{\mu}(M):=\{X\in{\mathfrak X}(M)\mid \operatorname{div}(X)=0\} 
\end{eqnarray}
be the space of divergence-free vector fields on $M$.
\begin{Lemma}
\label{Lemma-div-mu=Hmu}
	Let $M$ be a manifold
	with
	a volume form $\mu$
	and
	$H$ a nowhere vanishing function $H$ on $M$.
	Define another volume form on $M$ by
	$\mu_{H}:=H\mu$.
	Then,
	we have
	\begin{equation}
		\operatorname{div}_{\mu_{H}}(Y)
		=
		\operatorname{div}_{\mu}(HY).
	\end{equation}
	In particular,
	$Y\in{\mathfrak X}_{\mu_{H}}(M)$
	if and only if 
	$HY\in{\mathfrak X}_{\mu}(M)$.
\end{Lemma}
\begin{proof}
We have
\begin{eqnarray*}
	\operatorname{div}_{\mu_{H}}(Y)\mu_{H}
	&=&
	d\circ i_{Y}(\mu_{H})
	\\
	&=&
	d\circ i_{Y}(H\mu)
	\\
	&=&
	d\circ i_{HY}(\mu)
	\\
	&=&
	\operatorname{div}_{\mu}(HY)\mu,
\end{eqnarray*}
which completes the proof.
\end{proof}

Let $\operatorname{ev}_{p}:{\mathfrak X}(M)\to T_{p}M$
be the evaluation map,
where $T_{p}M$ is the tangent space of $M$ at $p\in M$.
Namely,
\begin{eqnarray}
	\operatorname{ev}_{p}:{\mathfrak X}(M)&\to& T_{p}M
	\label{eq-def-ev}\\
Y&\mapsto&\operatorname{ev}_{p}(Y):=Y_{p}.
\end{eqnarray}
We write 
${\mathfrak X}_{\mu,c}(M)$
for the space
of compactly supported divergence-free vector fields on $M$.

\begin{Lemma}
\label{Lemma-surj-Rk-ev}
Let $k\geq 2$
and $\mu_{0}$ the usual volume form on ${\mathbb R}^{k}$.
Then,
	for any $p\in {\mathbb R}^{k}$,
	the restriction of the evaluation map 
	$\operatorname{ev}_{p}:{\mathfrak X}_{\mu_{0},c}({\mathbb R}^{k}) \to T_{p}{\mathbb R}^{k}$
	is surjective.
\end{Lemma}
\begin{proof}
	We only prove in the case $k=2$
	and $p=(0,0)\in{\mathbb R}^{2}$.
	Let $\rho:{\mathbb R}\to {\mathbb R}$
	be a compactly supported smooth function
	satisfying $\rho\equiv 1$ on $[-1,1]$.
	Define functions $R_{1},R_{2}$
	by
	\begin{eqnarray*}
		R_{1}
		&:=&
		\sqrt{(x-1)^{2}+y^{2}},
		\\
		R_{2}
		&:=&
		\sqrt{x^{2}+(y-1)^{2}},
	\end{eqnarray*}
	for
	$(x,y)\in{\mathbb R}^{2}$.
	Moreover,
	we define vector fields
	\begin{eqnarray*}
		Y_{1}&:=&\rho(R_{1})
		\left(
		y\partial_{x}
		-
		(x-1)\partial_{y}
		\right),
		\\
		Y_{2}&:=&\rho(R_{2})
		\left(
		(y-1)\partial_{x}
		-
		x\partial_{y}
		\right).
	\end{eqnarray*}
	Recall that for a vector field $u=a\partial_{x}+b\partial_{y}$,
	we have
	\begin{eqnarray}
		\operatorname{div}_{\mu_{0}}(u)
		=
		\partial_{x}a +
		\partial_{y}b.
	\end{eqnarray}
	Thus,
	it is obvious that 
	$Y_{1},Y_{2}\in {\mathfrak X}_{\mu_{0},c}({\mathbb R}^{k})$
	and
	\begin{eqnarray}
		\operatorname{ev}_{(0,0)}(Y_{1})=\partial_{y},
		\quad
		\operatorname{ev}_{(0,0)}(Y_{1})=-\partial_{x}.
	\end{eqnarray}
	This completes the proof.
\end{proof}

\begin{Corollary}
\label{Cor-surjection-Rk-vect-div}
	Let $k\geq 2$
	and $\mu$ a volume form on ${\mathbb R}^{k}$.
	Then,
	for any $p\in {\mathbb R}^{k}$,
	the restriction of the evaluation map 
	$\operatorname{ev}_{p}:{\mathfrak X}_{\mu,c}({\mathbb R}^{k}) \to T_{p}{\mathbb R}^{k}$
	is surjective.
\end{Corollary}
\begin{proof}
	By the assumption,
	there exists a nowhere vanishing function $H$
	on ${\mathbb R}^{k}$
	such that 
	\begin{eqnarray}
		\mu=H\mu_{0}.
	\end{eqnarray}
	Then,
	$Y\in{\mathfrak X}_{\mu,c}(M)$
	if and only if 
	$HY\in{\mathfrak X}_{\mu_{0},c}(M)$
	for any $Y\in{\mathfrak X}(M)$
	by Lemma \ref{Lemma-div-mu=Hmu}.
	By definition,
	we have
	\begin{eqnarray}
		\operatorname{ev}_{p}(HY)
		=
		H(p)
		\operatorname{ev}_{p}(Y).
	\end{eqnarray}
	Moreover,
	$H(p)\neq 0$
	because $H$ is nowhere vanishing.
	Thus,
	we have
	$$
	\operatorname{ev}_{p}
	\left({\mathfrak X}_{\mu,c}({\mathbb R}^{k})
	\right)
	=
	\operatorname{ev}_{p}
	\left({\mathfrak X}_{\mu_{0},c}({\mathbb R}^{k})
	\right).
	$$
	Therefore,
	the corollary follows from Lemma \ref{Lemma-surj-Rk-ev}.
\end{proof}

\subsection{Lie group theory}
\label{Section-Lie-group}
In this section,
we recall 
some elementary theory of Lie groups.
For example,
see \cite{Geometry-of-Lie-groups}
for more details.

Let $G$ be a Lie group,
$G/H$ a homogeneous $G$-space,
and $G$ act on $G/H\times {\mathbb R}^{k}$
via the first factor.
We write ${\mathfrak X}(G/H\times {\mathbb R}^{k})^{G}$
for the space of $G$-invariant vector fields on 
$G/H\times {\mathbb R}^{k}$.
For any $gH\in G/H$,
let $\iota_{gH}$ be a closed embedding 
	\begin{eqnarray*}
		\iota_{gH}:{\mathbb R}^{k}
		&\to &
		 G/H\times {\mathbb R}^{k},
		\\
		x
		&\mapsto &
		\iota_{gH}(x):=(gH,x).
	\end{eqnarray*}
Then,
there exists 
an isomorphism
\begin{eqnarray}
	{\mathfrak X}({\mathbb R}^{k})
	&\simeq &  {\mathfrak X}(G/H\times {\mathbb R}^{k})^{G}
	\label{eq-isom-vect-G-inv}\\
	Y
		&\mapsto &
		\left(
		(gH,x)
		\mapsto
		\widetilde{Y}_{(gH,x)}
		:=
		\iota_{gH*}(Y_{x})
		\right),
		\nonumber
\end{eqnarray}
where $\iota_{gH*}$ is the pushforward by $\iota_{gH}$.
Moreover,
if $G$ is connected,
we have
\begin{equation}
	{\mathfrak X}(G/H\times {\mathbb R}^{k})^{G}
	=
	{\mathfrak X}(G/H\times {\mathbb R}^{k})^{\mathfrak g},
	\label{eq-isom-vect-G-g-connect}
\end{equation}
where ${\mathfrak g}$ is the Lie algebra of $G$
and 
we set
$$
{\mathfrak X}(G/H\times {\mathbb R}^{k})^{\mathfrak g}:=
\{
Y\in{\mathfrak X}(G/H\times {\mathbb R}^{k})\mid
[X,Y]=0\text{
for any
}
X\in{\mathfrak g}
\}.
$$
Here,
we identify $X\in{\mathfrak g}$
with the corresponding vector field on $G/H\times {\mathbb R}^{k}$.

Next,
we recall the top form version of the isomorphism \eqref{eq-isom-vect-G-inv}.
For this,
we suppose that there exists a $G$-invariant 
$n$-form $\omega_{0}$ on $G/H$,
where $n:=\dim G/H$.
Let 
\begin{equation}
p:G/H\times {\mathbb R}^{k}\to G/H,
\quad
q:G/H\times {\mathbb R}^{k}\to {\mathbb R}^{k}
\end{equation}
be the projections to each factor.
Then,
we have an isomorphism
\begin{eqnarray}
	{\mathcal E}^{k}({\mathbb R}^{k})
	&\simeq &  {\mathcal E}^{n+k}(G/H\times {\mathbb R}^{k})^{G}
	\label{eq-isom-top-form-G-inv}\\
	\mu
		&\mapsto &
		p^{*}(\omega_{0})\wedge
		q^{*}(\mu),
		\nonumber
\end{eqnarray}
where
${\mathcal E}^{k}({\mathbb R}^{k})$
is the space of $k$-th forms on ${\mathbb R}^{k}$
and 
$p^{*}$ is the pull back by $p$.

The following lemma states that
the isomorphism \eqref{eq-isom-vect-G-inv}
preserves 
the divergent 
if given volume form on $G/H\times {\mathbb R}^{k}$ 
is $G$-invariant.

\begin{Lemma}
\label{Lemma-isom-G-vect-inv-non-c}
	Let $G$ be a Lie group,
	$G/H$ a homogeneous $G$-space.
	Suppose that there exists a $G$-invariant $n$-form 
	$\omega_{0}$ on $G/H$ with $n:=\dim G/H$.
	Then,
	for any
	$G$-invariant $(n+k)$-form $\omega$ on $G/H\times {\mathbb R}^{k}$,
	there exists a $k$-form $\mu$ on ${\mathbb R}^{k}$
	such that
	the isomorphism \eqref{eq-isom-vect-G-inv}
	induces
	\begin{equation}
		{\mathfrak X}_{\mu}({\mathbb R}^{k})
	\simeq  {\mathfrak X}_{\omega}(G/H\times {\mathbb R}^{k})^{G}.
	\label{eq-Lemma-isom-vect-G-inv}
	\end{equation}
\end{Lemma}
\begin{proof}
Let $\omega$ be a
$G$-invariant $n$-form on $G/H$.
Then,
	by the isomorphism \eqref{eq-isom-top-form-G-inv},
	there exists a $k$-form $\mu$ on ${\mathbb R}^{k}$
	such that we have
	\begin{equation}
		\omega=p^{*}(\omega_{0})\wedge
		q^{*}(\mu).
	\end{equation}
Then,
for any $Y\in{\mathfrak X}_{\mu}({\mathbb R}^{k})$,
we have
\begin{eqnarray*}
	\operatorname{div}_{\omega}(\widetilde{Y})\omega
	&=&
	d\circ i_{\widetilde{Y}} 
	\left(
	p^{*}(\omega_{0})\wedge
		q^{*}(\mu)
	\right)
\end{eqnarray*}
by \eqref{eq-def-div-di}.
By the property of the interior derivative,
we have
\begin{eqnarray*}
	i_{\widetilde{Y}} 
	\left(
	p^{*}(\omega_{0})\wedge
		q^{*}(\mu)
		\right)
		&=&
	\left(
	i_{\widetilde{Y}} 
	p^{*}(\omega_{0})
	\right)	
	\wedge
		q^{*}(\mu)
		+
		(-1)^{n}
	p^{*}(\omega_{0})	
	\wedge
	\left(
	i_{\widetilde{Y}} 
		q^{*}(\mu)\right).
\end{eqnarray*}
By the definition of the pull back,
this is equal to
\begin{eqnarray}
	&=&
	p^{*}\left(
	i_{p_{*}\widetilde{Y}} 
	(\omega_{0})
	\right)	
	\wedge
		q^{*}(\mu)
		+
		(-1)^{n}
	p^{*}(\omega_{0})	
	\wedge
	q^{*}
	\left(
	i_{q_{*}\widetilde{Y}} 
		(\mu)
		\right).
		\label{eq-totyu-pqpq}
	\end{eqnarray}
	However,
	by the definition,
	we have
	\begin{eqnarray*}
		p_{*}(\widetilde{Y})=0,
		\quad
		q_{*}(\widetilde{Y})=Y.
	\end{eqnarray*}
	Thus,
	\eqref{eq-totyu-pqpq} is equal to
	\begin{eqnarray*}
		&=&
			(-1)^{n}
	p^{*}(\omega_{0})	
	\wedge
	q^{*}
	\left(
	i_{Y} 
		(\mu)
		\right).
	\end{eqnarray*}
	Thus,
	we have
	\begin{eqnarray*}
		\operatorname{div}_{\omega}(\widetilde{Y})\omega
		&=&
		d
		\left(
		(-1)^{n}
	p^{*}(\omega_{0})	
	\wedge
	q^{*}
	\left(
	i_{Y} 
		(\mu)
		\right)
		\right).
	\end{eqnarray*}
	By the property of the exterior derivative,
	this is equal to
	\begin{eqnarray*}
		&=&
		(-1)^{n}
		\left(
		d
	p^{*}(\omega_{0})	
	\right)
	\wedge
	q^{*}
	\left(
	i_{Y} 
		(\mu)
		\right)
		+
		(-1)^{2n}
	p^{*}(\omega_{0})	
	\wedge
	\left(
		d
	q^{*}
	\left(
	i_{Y} 
		(\mu)
		\right)
			\right)
			\\
			&=&
			p^{*}(\omega_{0})	
	\wedge
	\left(
		d
	q^{*}
	\left(
	i_{Y} 
		(\mu)
		\right)
			\right).
	\end{eqnarray*}
	Because $d$ and the pull back is commutative,
	we have
	\begin{eqnarray*}
		\operatorname{div}_{\omega}(\widetilde{Y})\omega
		&=&
		p^{*}(\omega_{0})	
	\wedge
	\left(
		q^{*}\left(d
		\circ
	i_{Y} 
		(\mu)
		\right)
			\right)
			\\
			&=&
			p^{*}(\omega_{0})	
	\wedge
	\left(
		q^{*}\left(\operatorname{div}_{\mu}(Y)
		\mu
		\right)
			\right)
			\\
			&=&
			\operatorname{div}_{\mu}(Y)
			\left(
			p^{*}(\omega_{0})	
	\wedge
		q^{*}\left(
		\mu
		\right)
		\right)
		\\
		&=&
		\operatorname{div}_{\mu}(Y)
		\omega.
	\end{eqnarray*}
	This completes the proof.
\end{proof}

\subsection{$S^{1}$ case}
In this section,
we apply the results of Section \ref{Section-Lie-group}
to the case $G:=S^{1}$.

Let $N$ be a one-dimensional $G$-space.
Then,
it is obvious that there exists 
a $G$-invariant 1-form $\omega_{0}$ on $N$.
Therefore we have the following.
\begin{Lemma}
\label{Lemma-isom-S1-inv-vecto-div-c}
	Let $N$ be a one-dimensional $S^{1}$-homogeneous space.
	Then,
	for any $S^{1}$-invariant $(k+1)$-from $\omega$ on $N\times {\mathbb R}^{k}$,
	there exists a $k$-form $\mu$ on ${\mathbb R}^{k}$
	such that 
	we have an isomorphism
	\begin{eqnarray*}
		{\mathfrak X}_{\mu,c}({\mathbb R}^{k})
	\simeq  {\mathfrak X}_{\omega,c}(G/H\times {\mathbb R}^{k})^{S^{1}}.
	\end{eqnarray*}
\end{Lemma}
\begin{proof}
We note that 
the isomorphism 
\eqref{eq-Lemma-isom-vect-G-inv}
preserves the compact support property
because $S^{1}$ is compact.
Thus,
	this is an immediate corollary of Lemma \ref{Lemma-isom-G-vect-inv-non-c}.
\end{proof}

\begin{Corollary}
\label{Cor-isom-X-inv}
Let $N$ be a one-dimensional $S^{1}$-homogeneous space,
and $X$ a vector field on $N\times {\mathbb R}^{k}$
induced by the $S^{1}$-action.
	Then,
	for any $S^{1}$-invariant $(k+1)$-from $\omega$ on $N\times {\mathbb R}^{k}$,
	there exists a $k$-form $\mu$ on ${\mathbb R}^{k}$
	such that 
	we have an isomorphism
\begin{equation}
	{\mathfrak X}_{\mu,c}({\mathbb R}^{k})
	\simeq   {\mathfrak X}_{\omega,c}(G/H\times {\mathbb R}^{k})^{X},
\end{equation}
where
we set
\begin{equation}
	 {\mathfrak X}_{\omega,c}(G/H\times {\mathbb R}^{k})^{X}
	 :=
	 \{
Y\in{\mathfrak X}_{\omega,c}(G/H\times {\mathbb R}^{k})\mid
[X,Y]=0
\}.
\end{equation}
\end{Corollary}
\begin{proof}
	Because $S^{1}$ is connected,
	we have
	\begin{eqnarray*}
		{\mathfrak X}_{\omega,c}(G/H\times {\mathbb R}^{k})^{S^{1}}
		=
		{\mathfrak X}_{\omega,c}(G/H\times {\mathbb R}^{k})^{\mathfrak g},
	\end{eqnarray*}
	where ${\mathfrak g}$ is the Lie algebra of $S^{1}$
	by \eqref{eq-isom-vect-G-g-connect}.
	Moreover,
	we have
	${\mathfrak g}={\mathbb R}X$
	because $S^{1}$ is one-dimensional.
	Thus,
	by definition,
	we have
	$$
	{\mathfrak X}_{\omega,c}(G/H\times {\mathbb R}^{k})^{\mathfrak g}
	=
	{\mathfrak X}_{\omega,c}(G/H\times {\mathbb R}^{k})^{X}.
	$$
	This completes the proof by
	Lemma \ref{Lemma-isom-S1-inv-vecto-div-c}.
\end{proof}

Recall that the evaluation map 
$\operatorname{ev}_{p}$
defined in \eqref{eq-def-ev}.

\begin{Corollary}
\label{Cor-surjection-ev-G-inv}
	Let $N$ be a one-dimensional $S^{1}$-homogeneous space,
	$\omega$ a $S^{1}$-invariant $(k+1)$-form on $N\times {\mathbb R}^{k}$,
	$X$ a vector field on $N\times {\mathbb R}^{k}$
	induced by the $S^{1}$-action.
	Then,
	we have
	\begin{eqnarray*}
		T_{p}(N\times {\mathbb R}^{k})
		=
		{\mathbb R}X
		+
		\operatorname{ev}_{p}
		\left(
		{\mathfrak X}_{\omega,c}(G/H\times {\mathbb R}^{k})^{X}
		\right)
	\end{eqnarray*}
	for any $p\in N\times {\mathbb R}^{k}$.
\end{Corollary}

\begin{proof}
Set $G:=S^{1}$
and write $N\simeq G/H$
and
$p=(gH,x)\in G/H\times {\mathbb R}^{k}$.
Then,
	for any $Y\in{\mathfrak X}({\mathbb R}^{k})$,
	we have
	\begin{eqnarray*}
		\operatorname{ev}_{p}(\widetilde{Y})
		&=&
		\widetilde{Y}_{(gH,x)}
		\\
		&=&
		\iota_{gH*}(Y_{x})
	\end{eqnarray*}
	by definition (see \eqref{eq-isom-vect-G-inv}).
	Thus,
	we have
	\begin{eqnarray*}
	\operatorname{ev}_{p}
		\left(
		{\mathfrak X}_{\omega,c}(G/H\times {\mathbb R}^{k})^{X}
		\right)
		&=&
		\iota_{gH*}
		\left(
		\operatorname{ev}_{x}({\mathfrak X}_{\mu,c}({\mathbb R}^{k}))
		\right)
		\\
		&=&
		\iota_{gH*}(T_{x}{\mathbb R}^{k})
	\end{eqnarray*}
	by Corollaries \ref{Cor-surjection-Rk-vect-div} and 
	\ref{Cor-isom-X-inv}.
Because 
\begin{eqnarray*}
	T_{p}(N\times {\mathbb R}^{k})
	=
	{\mathbb R}X
	+
	\iota_{gH*}
		\left(
		T_{x}{\mathbb R}^{k}
		\right),
\end{eqnarray*}
we have the corollary.
\end{proof}

\subsection{Proof of Lemma \ref{Lemma-existence-Y-nonzero}}
We rewrite Lemma \ref{Lemma-existence-Y-nonzero}
using the notation 
introduced in this section.
\begin{Lemma}
	Let $N$ be a a one-dimensional homogeneous $S^{1}$-space,
	$\omega$ be a $S^{1}$-invariant volume form on $N\times {\mathbb R}^{k}$,
	$X$ a vector field on $N\times {\mathbb R}^{k}$
	induced by the $S^{1}$-action,
	and $h$ a non-constant function on $N\times {\mathbb R}^{k}$
	satisfying $X(h)=0$.
	Then, if $k\geq 2$, there exists
	$Y\in {\mathfrak X}_{\omega,c}(N\times {\mathbb R}^{k})^{X}$ 
	such that $Y(h)\neq 0$.
\end{Lemma}
\begin{proof}
We suppose that any
$Y\in {\mathfrak X}_{\omega,c}(N\times {\mathbb R}^{k})^{X}$
satisfies $Y(h)= 0$.
Then,
because $X(h)=0$,
$h$ is killed by all the derivatives
by Corollary \ref{Cor-surjection-ev-G-inv}.
This contradicts to the non-constant assumption of $h$.
Thus,
we have the lemma.
\end{proof}

\section{Diffeomorphism group and Misio{\l}ek curvature}
\label{Appendix-diffomorphism-and-MC}

In this section,
we recall
the theory of
diffeomorhphism groups
in the context of 
inviscid fluid flows
for the completeness.
Our main references are
\cite{EMa} and \cite{Mstability}.

Let $(M, g)$ be a compact $n$-dimensional Riemannian manifold without boundary
and ${\mathcal D}^{s}(M)$  
the group of Sobolev $H^{s}$ diffeomorphisms of $M$
and
${\mathcal D}_{\mu}^{s}(M)$
the subgroup of ${\mathcal D}^{s}(M)$
consisting volume preserving elements,
where $\mu$ is the volume form on $M$ defined by $g$.
If $s>1+\frac{n}{2}$,
the group ${\mathcal D}^{s}(M)$
can be 
given a structure of an infinite-dimensional weak Riemannian manifold
(see \cite{EMa})
and ${\mathcal D}^{s}_{\mu}(M)$ become its
weak Riemannian submanifold
(The term ``weak''
means that
the topology induced from the metric 
is weaker than
the original topology of
${\mathcal D}^{s}(M)$
or 
${\mathcal D}^{s}_{\mu}(M)$).
This weak Riemannian metric is given as follows:
The tangent space $T_{\eta}{\mathcal D}^{s}(M)$ 
of ${\mathcal D}^{s}(M)$ at a 
point $\eta\in{\mathcal D}^{s}(M)$ consists of all 
$H^{s}$ vector fields on $M$ which cover $\eta$, namely, 
all $H^{s}$ sections of the pullback bundle $\eta^{*}TM$.
Thus
for $x\in M$
and 
$V, W\in T_{\eta}{\mathcal D}^{s}(M)$,
we have 
$V(x),
W(x)
\in
T_{\eta(x)}M$.
Then we define an inner product on $T_{\eta}{\mathcal D}^{s}(M)$
by
\begin{eqnarray}
\langle V, W\rangle
:=
\int_{M}
g(V(x),W(x))\mu(x)
\label{L2metric}
\end{eqnarray}
and set $|V|:=\sqrt{\langle V, V\rangle}$.
Similarly,
$T_{\eta}{\mathcal D}^{s}_{\mu}(M)$ 
consists of all 
$H^{s}$ divergence-free vector fields on $M$ which cover 
$\eta\in{\mathcal D}_{\mu}^{s}(M)$.
Therefore, the metric \eqref{L2metric} induces a direct sum:
\begin{eqnarray}
T_{\eta}{\mathcal D}^{s}(M)=
T_{\eta}{\mathcal D}^{s}_{\mu}(M)
\oplus
\{
({\rm grad} f)
\circ 
\eta
\mid
f\in
H^{s+1}(M)\},
\label{directsum}
\end{eqnarray}
which follows from the fact that
the gradient
is 
the adjoint of 
the negative divergence.
Let
\begin{eqnarray*}
	P_{\eta}&:& T_{\eta}{\mathcal D}^{s}(M)
	\to T_{\eta}{\mathcal D}^{s}_{\mu}(M)
	\\
	Q_{\eta}&:& T_{\eta}{\mathcal D}^{s}(M)
	\to 
	\{
({\rm grad} f)
\circ 
\eta
\mid
f\in
H^{s+1}(M)\}
\end{eqnarray*}
be the projection to the first and second components of \eqref{directsum},
respectively.
Moreover,
we write $e\in{\mathcal D}^{s}(M)$ for the 
identity element of 
${\mathcal D}^{s}(M)$.

The metric \eqref{L2metric}
also
induces
the
right invariant Levi-Civita connections $\bar{\nabla}$ and
$\widetilde{\nabla}$ on ${\mathcal D}^{s}(M)$ and
${\mathcal D}^{s}_{\mu}(M)$,
respectively.
This is defined as follows:
Let 
$V, W$ be vector fields on ${\mathcal D}^{s}(M)$.
We write
$V_{\eta}\in T_{\eta}{\mathcal D}^{s}(M)$
for
the value of 
$V$
at $\eta\in {\mathcal D}^{s}(M)$.
Then we have $V_{\eta}\circ \eta^{-1}, W_{\eta}\circ \eta^{-1}\in T_{e}{\mathcal D}^{s}(M)$,
namely,
$V_{\eta}\circ \eta^{-1}$ and $W_{\eta}\circ \eta^{-1}$ are vector fields on $M$.
Moreover,
we have
$W_{\eta}\circ \eta^{-1}$ is a vector field of class $C^{1}$ on $M$
by Sobolev embedding theorem
and the assumption $s>1+\frac{n}{2}$.
Thus we can consider $\nabla_{V_{\eta}\circ \eta^{-1}}W_{\eta}\circ \eta^{-1}$,
where $\nabla$ is the Levi-Civita connection on $M$.
Take a path $\varphi$ on ${\mathcal D}^{s}(M)$ satisfying
$\varphi(0)=\eta$ and $V_{\eta}=\partial_t\varphi(0)\in T_{\eta}{\mathcal D}_{\mu}^{s}(M)$,
then we define
\begin{eqnarray}
(\bar\nabla_{V}W)_{\eta}
&:=&
\frac{d}{dt}\left(
W_{\varphi(t)}\circ\varphi^{-1}(t)
\right)|_{t=0}\circ\eta+(\nabla_{V_{\eta}\circ\eta^{-1}}
W_{\eta}\circ\eta^{-1})\circ\eta.
\label{nab1}
\end{eqnarray}
Moreover, if $V$ and $W$ are vector fields on 
${\mathcal D}^{s}_{\mu}(M)$,
we define
\begin{eqnarray}
(\widetilde{\nabla}_{V}W)_{\eta}
&:=&P_\eta(\bar{\nabla}_{V}W)_{\eta}
\label{nab2}.
\end{eqnarray}
These definitions are independent of the particular choice of $\varphi(t)$.
We note that
$(\bar\nabla_{V}W)_{\eta}=(\bar\nabla_{V}W)_{e}\circ \eta$
if $V$ and $W$ are right invariant
vector fields on ${\mathcal D}^{s}(M)$
(i.e., $\bar{\nabla}$ is right invariant).
This is because
if $W$ is right invariant,
or equivalently, if $W$ satisfies $W_{\eta}=W_{e}\circ \eta$ for any $\eta\in{\mathcal D}^{s}_{\mu}(M)$,
the first term of \eqref{nab1} vanishes.

A geodesic joining the identity element $e\in {\mathcal D}^{s}_{\mu}(M)$ and
$p\in{\mathcal D}^{s}_{\mu}(M)$ 
can be obtained 
from
a variational principle as a stationary point of the energy function:
\begin{eqnarray}
E(\eta)^{t_{0}}_{0}:=\frac{1}{2}
\int_{0}^{t_{0}}
\left|
\dot{\eta}(t)
\right|^{2}dt,
\label{Eint}
\end{eqnarray}
where $\eta$ is a curve on ${\mathcal D}^{s}_{\mu}(M)$
satisfying $\eta(0)=e$ and $\eta(t_{0})=p$
and we set $\dot{\eta}(t):=\partial_{t}\eta(t)\in T_{\eta(t)}{\mathcal D}^{s}_{\mu}(M)$.
Let $\xi(r,t):(-\varepsilon,\varepsilon)\times[0,t_{0}]\to {\mathcal D}^{s}_{\mu}(M)$
be
a variation of a geodesic $\eta(t)$ with fixed end points,
namely,
it satisfies
$\xi(r,0) = \eta(0)$, $\xi(r,t_{0}) = \eta(t_{0})$ and $\xi(0,t)=\eta(t)$ for $t\in[0,t_{0}]$.
We sometimes write $\xi_{r}(t)$ for $\xi(r,t)$.
Let
$X(t):=\partial_{r}\xi(r,t)|_{r=0}\in T_{\eta(t)}{\mathcal D}^{s}_{\mu}(M)$ be the associated vector field
on ${\mathcal D}^{s}_{\mu}(M)$. 
Then 
the first and the second variations of the above integral are given by
\begin{alignat}{1}
0=
E'(\eta)_{0}^{t_{0}}(X)
=&
(X(t_{0}),\dot{\eta}(t_{0}))
-
(X(0),\dot{\eta}(0))\nonumber\\
&-
\int_{0}^{t_{0}}
(X(t),
\widetilde{\nabla}
_{\dot{\eta}(t)}\dot{\eta}(t)
)
dt,
\nonumber\\
E''(\eta)^{t_0}_{0}(X,X)
=&
\int_{0}^{t_{0}}
\{
(
\widetilde{\nabla}_{\dot{\eta}} X,
\widetilde{\nabla}_{\dot{\eta}} X)
-
(
\widetilde{R}_{\eta}
(X, \dot{\eta}
)\dot{\eta},
X)
\}dt.
\label{E''}
\end{alignat}

The reason why the geometry of ${\mathcal D}_{\mu}^{s}(M)$ 
is important
is 
that
geodesics in ${\mathcal D}^{s}_{\mu}(M)$ 
correspond to inviscid fluid flows on $M$,
which was first remarked by V. I. Arnol'd \cite{A}.
This correspondence is accomplished in the following way:
If $\eta(t)$ is a geodesic on ${\mathcal D}^{s}_{\mu}(M)$ (i.e., $\widetilde{\nabla}_{\dot{\eta}}\eta=0$)
joining $e$ and $\eta(t_{0})$,
a time dependent vector field on $M$ defined by
$u(t):=\dot{\eta}(t)\circ \eta^{-1}(t)$ 
is a solution to the Euler equations on $M$:
\begin{align} \nonumber 
&\partial_tu + \nabla_{u} u = - \operatorname{grad}p
\qquad 
t \in [0,t_{0}], 
\\ 
\label{Epre} 
&\operatorname{div} u=0,
\\ 
\nonumber 
&u|_{t=0} = \dot{\eta}(0), 
\end{align} 
with a scalar  function (pressure) $p(t)$ determined by $u(t)$.
Here
$\operatorname{grad}p$ (resp. $\operatorname{div}u$) is the gradient (resp.\:divergent) of $p$ (resp.\:$u$) with respect to the Riemannian metric $g$ of $M$.
In this context,
the existence of conjugate points along a geodesic $\eta$ corresponds to the Lagrangian stability of a fluid flow $u=\dot{\eta}\circ\eta^{-1}$.

In this context,
G. Misio{\l}ek essentially established
a criterion  
for the existence of a conjugate point
on a geodesic corresponding to a stationary solution of 
in \cite{MconjT2},
which we call M-criterion.
\begin{Definition}[{\cite{MconjT2}, \cite[(2.14)]{TY1}, \cite[Lem.~B.6]{TY2}}]
	Let $u$ be a stationary solution of \eqref{Epre}
and $v$ a divergence-free vector field on $M$.
Then,
the Misio{\l}ek curvature ${\mathfrak m}{\mathfrak c}_{u,v}$ is defined by
\begin{eqnarray*}
	{\mathfrak m}{\mathfrak c}_{u,v}
	&:=&
	\langle \nabla_{u}[u,v]+
	\nabla_{[u,v]}v,v\rangle
	\\
	&=&
	-|[u,v]|^{2}
	-
	\langle
	u,
	[[u,v],v]
	\rangle.
\end{eqnarray*}
\end{Definition}

\begin{Fact}[{\cite[Lems.~2 and 3]{MconjT2}, \cite[Fact.~1.1]{TY1}}]
\label{Fact-M-criterion}
Let $M$ be a compact $n$-dimensional 
Riemannian manifold without boundary
and $s>2+\frac{n}{2}$.
Suppose that $u\in T_{e}{\mathcal D}^{s}_{\mu}(M)$
is a stationary solution of \eqref{Epre} on $M$
and take a geodesic $\eta(t)$
on ${\mathcal D}^{s}_{\mu}(M)$
satisfying $u=\dot{\eta}\circ\eta$.
Then,
if $v\in T_{e}{\mathcal D}^{s}_{\mu}(M)$ satisfies
${\mathfrak m}{\mathfrak c}_{u,v}>0$,
then,
there exists a point conjugate to $e\in{\mathcal D}^{s}_{\mu}(M)$
along $\eta(t)$
on $0\leq t\leq t^{*}$
for some $t^{*}>0$.
\end{Fact}

\end{document}